\documentclass[10pt]{amsart}
\usepackage{amsmath}
\usepackage{amsfonts}
\usepackage{amsthm}
\usepackage{enumerate}
\usepackage{amssymb}
\usepackage{amscd}
\usepackage[all]{xy}

  \let\leq=\leqslant
  \let\geq=\geqslant

\DeclareMathOperator{\gr}{gr}
\DeclareMathOperator{\Aut}{Aut}
\DeclareMathOperator{\Ann}{Ann}
\DeclareMathOperator{\End}{End}
\DeclareMathOperator{\Hom}{Hom}

\DeclareMathOperator{\GL}{GL}
\DeclareMathOperator{\SL}{SL}

\DeclareMathOperator{\Sat}{Sat}
\DeclareMathOperator{\Sym}{Sym}
\DeclareMathOperator{\pd}{pd}

\DeclareMathOperator{\ass}{ass}
\DeclareMathOperator{\udim}{udim}
\DeclareMathOperator{\Spec}{Spec}

\newcommand{\be}{\begin{enumerate}[{(}a{)}]}
\newcommand{\ee}{\end{enumerate}}
\newcommand{\mb}[1]{\quad\mbox{#1}\quad}
\begin{document}
\theoremstyle{plain}

\newtheorem{MainThm}{Theorem}
\renewcommand{\theMainThm}{\Alph{MainThm}}

\newtheorem*{trm}{Theorem}
\newtheorem*{lem}{Lemma}
\newtheorem*{prop}{Proposition}
\newtheorem*{defn}{Definition}
\newtheorem*{thm}{Theorem}
\newtheorem*{example}{Example}
\newtheorem*{cor}{Corollary}
\newtheorem*{conj}{Conjecture}
\newtheorem*{hyp}{Hypothesis}
\newtheorem*{thrm}{Theorem}
\newtheorem*{quest}{Question}
\newtheorem*{rems}{Remarks}

\newcommand{\Fp}{\mathbb{F}_p}
\newcommand{\Zp}{\mathbb{Z}_p}
\newcommand{\Qp}{\mathbb{Q}_p}
\newcommand{\Kr}{\mathcal{K}}
\newcommand{\invlim}{\lim\limits_{\longleftarrow}}
\newcommand{\wt}[1]{\langle {#1}, \omega(\mathbf{g}) \rangle}
\newcommand{\wth}[1]{w(\mathbf{c}^{#1})}
\newcommand{\Uhat}[1]{\widehat{U_\mathfrak{#1}}}
\newcommand{\qdel}[1]{\partial_{\mathbf{g}}^{(#1)}}
\newcommand{\C}[1]{\mathbf{c}^{#1}}
\title{Prime ideals in nilpotent Iwasawa algebras}
\author{Konstantin Ardakov}
\address{\newline University of Nottingham, \newline School of Mathematical Sciences, \newline University Park, \newline Nottingham NG7 2RD \newline }
\email{\newline konstantin.ardakov@gmail.com}
\begin{abstract} Let $G$ be a nilpotent complete $p$-valued group of finite rank and let $k$ be a field of characteristic $p$. We prove that every faithful prime ideal of the Iwasawa algebra $kG$ is controlled by the centre of $G$, and use this to show that the prime spectrum of $kG$ is a disjoint union of commutative strata. We also show that every prime ideal of $kG$ is completely prime. The key ingredient in the proof is the construction of a non-commutative valuation on certain filtered simple Artinian rings.
\end{abstract}
\setcounter{tocdepth}{1}
\date{\today}
\thanks{This research was supported by an Early Career Fellowship from the Leverhulme Trust.}
\maketitle
\tableofcontents

\section{Introduction}

\subsection{Prime ideals and Iwasawa algebras} Let $G$ be a compact $p$-adic analytic group and let $k$ be a field of characteristic $p$. The completed group algebra $kG$ of $G$ with coefficients in $k$, also known as an Iwasawa algebra, is an interesting example of a non-commutative Noetherian complete semilocal ring with good homological properties --- see the survey article \cite{ArdBro2006} for an introduction to this area. A long-running project aims to understand the prime spectrum $\Spec(kG)$ of $kG$, guided in part by the list of open questions in $\S 6$ of this survey paper. Progress so far has been rather limited: the strongest known result to date, \cite[Theorem 4.8]{ArdWad2009} asserts that (under mild restrictions on the prime $p$) when the Lie algebra $\mathfrak{g}$ of $G$ is split semisimple, the homological height of a non-zero prime ideal in $kG$ is bounded below by an integer $u$ depending only $\mathfrak{g}$; for example if $\mathfrak{g} = \mathfrak{sl}_n(\Qp)$ then $u = 2n-2$. 

\subsection{Complete $p$-valued groups}Lazard proved in 1965 that it is always possible to find a closed normal subgroup $N$ of finite index in $G$ with particularly nice properties. For any such subgroup there is a crossed product decomposition $kG = kN \ast (G/N)$, and the going-up and going-down theorems \cite{Pass} for such crossed products give a strong connection between $\Spec(kG)$ and $\Spec(kN)$. Because of this it is important to first better understand $\Spec(kN)$. Typically one can choose $N$ to be a uniform pro-$p$ group, but it will be more convenient for us to work with a slightly larger class of groups --- Lazard's \emph{complete $p$-valued groups of finite rank}. See $\S\ref{pVal}$ for the precise definitions. 

\subsection{Construction of the `obvious' prime ideals}  Let, then, $G$ be a complete $p$-valued group of finite rank. Currently, the only known way to obtain two-sided ideals in $kG$ is to either take a centrally generated ideal, to induce up from a closed normal subgroup or to take an inverse image ideal. Let us make this more precise. 

We say that a prime ideal $P$ is \emph{faithful} if $G$ embeds faithfully into the group of units of $kG/P$, or equivalently, if $P^\dag := (1 + P)\cap G$ is the trivial group. Let
\[ \Spec^f(kG)\]
denote the set of all faithful prime ideals of $kG$. If $N$ is a closed normal subgroup of $G$, we say that $N$ is \emph{isolated} if $G/N$ is a torsion-free group, and we will write $N\triangleleft^i_c G$ to denote this. We show in Lemma \ref{IsoSub} that $P^\dag \triangleleft^i_c G$ for any $P \in \Spec(kG)$.

Let $N \triangleleft^i_c G$ and let $Z_N = \widetilde{N} / N$ denote the centre of $G/N$; this is a free abelian pro-$p$ group of finite rank $d_N\geq 0$ say. Then the algebra $kZ_N$ is just a commutative formal power series ring in $d_N$ variables over $k$. Now if $Q$ is a faithful prime ideal of $kZ_N$, let $\widetilde{Q}$ be its preimage in $k\widetilde{N}$ and let $\widetilde{Q}kG$ be its extension to $kG$. It follows from Theorem \ref{CompPrime} below that $\widetilde{Q}kG$ is always a prime ideal in $kG$, and in this way we obtain a map
\[ \begin{array}{ccccc}\Theta &:& \coprod\limits_{N \triangleleft_c^i G} \Spec^f(kZ_N) &\to& \Spec(kG) \\
&&Q &\mapsto & \widetilde{Q}kG .\end{array}
\]
There is a natural bijection between the set of closed isolated normal subgroups of $G$ and the set of ideals of the Lie algebra $\mathfrak{g}$ of $G$.
\subsection{A partial inverse map to $\Theta$}
Let $P$ be a prime ideal in $kG$. Then $P \cap k\widetilde{P^\dag}$ is a prime ideal in $k\widetilde{P^\dag}$ containing $P^\dag - 1$ because $\widetilde{P^\dag}$ is central modulo $P^\dag$, so we obtain a prime ideal
\[ \Psi(P) := \frac{ P \cap k\widetilde{P^\dag} }{(P^\dag - 1)k\widetilde{P^\dag}}\]
of $kZ_{P^\dag}$. It is easy to see that $\Psi(P)$ is faithful, and in this way we obtain a map
\[ \Psi : \Spec(kG) \to \coprod\limits_{N \triangleleft_c^i G} \Spec^f(kZ_N).\]
Here is our main result:
\begin{MainThm}\label{MainA}Let $G$ be a complete $p$-valued group of finite rank. Then
\be
\item $\Psi(\Theta(Q)) = Q$ for any $N \triangleleft_c^i G$ and $Q \in \Spec^f(kZ_N)$, and
\item $\Theta(\Psi(P)) = P$ for any $P \in \Spec(kG)$, whenever $G$ is nilpotent.
\ee 
Every ideal of the form $\Theta(Q)$ is completely prime.
\end{MainThm}
Thus $\Spec(kG)$ always contains the disjoint union of the ``commutative strata" $\Theta\left(\Spec^f(kZ_N)\right)$ and is actually equal to this union when $G$ is nilpotent. In fact, the evidence we have so far leads us to suspect that this assumption on $G$ is not necessary. The proof is given in $\S \ref{PfMainA}$.
\subsection{Zalesskii's Theorem}\label{ContSubIntro} Let $I$ be a right ideal in $kG$. We say that a closed subgroup $U$ of $G$ \emph{controls} $I$ if and only if $I = (I \cap kU)kG$. In $\S 2.7$ of the companion paper \cite{Ard2011}, we defined the \emph{controller subgroup} of $I$ to be the intersection $I^\chi$ of all \emph{open} subgroups of $G$ that control $I$:
\[I^\chi = \bigcap \{U \leq_o G : I = (I \cap kU)kG \}.\]
It follows from \cite[Theorem A]{Ard2011} that a closed subgroup $H$ of $G$ controls $I$ if and only if $H \supseteq I^\chi$. In particular, $I^\chi$ itself controls $I$, and $(I \cap kI^\chi)^\chi = I^\chi$. 

Let $Z$ be the centre of $G$. The real content of Theorem \ref{MainA}, namely part (b), quickly follows from Theorem \ref{Zal} which asserts that 
\[\mbox{\emph{if $G$ is nilpotent, then every faithful prime ideal $P$ of $kG$ is controlled by $Z$,}}\]
or equivalently, that $G$ must act trivially on $P^\chi$ by conjugation. Of course this is a direct analogue of Zalesskii's Theorem on prime ideals in group algebras of nilpotent groups --- see \cite{Zal}. Theorem \ref{Zal} in turn follows from our main technical result, namely

\begin{MainThm}\label{MainB} Let $G$ be a complete $p$-valued group of finite rank and let $P$ be a faithful prime ideal of $kG$. Let $\varphi$ be a non-trivial automorphism in $\Aut^\omega_Z(G)$ such that $\varphi(P) = P$. Then $P$ is controlled by some proper closed subgroup $H$ of $G$.
\end{MainThm}
Here $\Aut^\omega_Z(G)$ is a certain ``small" group of automorphisms of $G$ that act trivially modulo $Z$ --- see $\S \ref{AutModZ}$ for the precise definition. The deduction of Theorem \ref{Zal} from Theorem \ref{MainB} is performed in $\S \ref{Gamma}$; this is not entirely trivial because $P \cap kP^\chi$ need not in general be a prime ideal of $kP^\chi$. Theorem \ref{MainB} can also be viewed as a non-commutative analogue of Roseblade's \cite[Theorem D]{Roseblade}.

\subsection{The strategy of the proof} To prove Theorem \ref{MainB}, we let $\tau : kG \to Q$ be the natural map from $kG$ to the classical ring of quotients $Q$ of the prime Noetherian ring $kG/P$, and consider certain Mahler expansions
\[\tau \varphi^{p^r} = \sum_{\alpha \in \mathbb{N}^d} \tau \left(\langle \varphi^{p^r}, \qdel{\alpha} \rangle\right) \cdot \tau \qdel{\alpha} \quad\mbox{for all}\quad r \geq 0\]
inside the vector space of all $k$-linear maps from $kG$ to $Q$ --- see Corollary \ref{Extend} and $\S \ref{ModPExp}$. We study the growth rates of the Mahler coefficients $\tau \left(\langle \varphi^{p^r}, \qdel{\alpha} \rangle\right)$ as $r \to \infty$ and define an appropriate $Q$-linear combination
\[\zeta^{(i)}_r := \sum_{j=1}^m (M_r^{-1})_{ij} (\tau\varphi^{p^{r + j-1}} - \tau)\]
of these $\tau \varphi^{p^r}$. On the one hand, each of these operators sends $P$ to zero since $\varphi$ preserves $P$. On the other hand, we show in Theorem \ref{Zeta} that the limit of $\zeta^{(i)}_r$ as $r \to \infty$ equals one of the ``quantized derivations" 
\[\tau \partial_i : kG \to Q.\]
This is enough to deduce Theorem \ref{MainB} --- see $\S \ref{PfCntrlThm}$ below.

\subsection{A key ingredient}
In order to make sense of $\lim_{r \to \infty} \zeta^{(i)}_r$ and to construct the ``correct" $\zeta^{(i)}_r$, we need to equip the simple Artinian ring $Q$ with a well-behaved filtration. This is obtained from
\begin{MainThm}\label{MainC} Let $R$ be a prime ring and let $w : R \to \mathbb{Z} \cup \{\infty\}$ be a Zariskian filtration. Suppose that $F := R_0/R_1$ is a field and that $\gr R$ is a commutative infinite-dimensional $F$-algebra. Then there exists a filtration $v : Q \to \mathbb{Z}\cup \{\infty\}$ on the classical ring of quotients $Q$ of $R$ and a central simple algebra $C$, such that
\begin{enumerate}[{(}a{)}] 
\item the natural inclusion $(R, w) \to (Q,v)$ is continuous,
\item if $w(x) \geq 0$ then $v(x) \geq 0$, and
\item $\gr Q \cong C[X,X^{-1}]$.
\end{enumerate}
Moreover, restriction of $v$ to the centre of $Q$ is a valuation.
\end{MainThm}
Even though $R$ itself is prime, the associated graded ring $\gr R$ with respect to the original filtration $w$ is in general not prime; worse still, it could contain non-zero nilpotent elements. For an example of such behaviour, consider the (commutative!) ring $R = k[[x,y]] / \langle x^2 - y^3 \rangle$ equipped with the $\langle x,y \rangle$-adic filtration. Theorem \ref{MainC} shows that under rather mild hypotheses it is always possible ``improve" this filtration to one whose associated graded ring is as nice as one could possibly hope for. Perhaps our $v$ deserves to be called a ``non-commutative valuation".

We hope that Theorem \ref{MainC} will be of independent interest, since it is applicable to prime factor rings of not only Iwasawa algebras, but also universal enveloping algebras of finite dimensional Lie algebras. It is proved in $\S \ref{ConstNCVal}$.

\subsection{Another application}\label{JustInfIntro} Let $G$ be a compact $p$-adic analytic group. We say that a finitely generated $kG$-module $M$ is \emph{just infinite} if $M$ is infinite dimensional over $k$ but $M/N$ is finite dimensional over $k$ for every non-zero $kG$-submodule $N$ of $M$. Equivalently $M$ is a critical $kG$-module of Krull dimension $1$. 

Using Theorem \ref{MainB} we can give an example of a just infinite ``parabolic Verma module" for $kG$. We do not strive for the maximal generality here, and just wish to illustrate the method.

\begin{MainThm}\label{MainD} Let $G$ be the second congruence subgroup of $\SL_n(\Zp)$, let $\mathfrak{p}$ be a maximal parabolic subalgebra of $\mathfrak{g} = \mathfrak{sl}_n(\Qp)$, and let $P = \exp( \mathfrak{p} \cap \log(G) )$ be the corresponding uniform subgroup of $G$. Then the induced module $k \otimes_{kP} kG$ is just infinite.
\end{MainThm}
The proof is given in $\S \ref{ParaVerma}$. This result can be viewed as further (rather weak) evidence for the Krull dimension conjecture --- see \cite[Question D]{ArdBro2006}. Note also that $k \otimes_{kP} kG$ can be arbitrarily ``large", since its canonical dimension $\dim \mathfrak{g}/\mathfrak{p}$ has no upper bound as $n$ increases.

\subsection{Acknowledgements} This research was supported by an Early Career Fellowship from the Leverhulme Trust. I am very grateful to the Trust for giving me the opportunity to focus on the problem of prime ideals in Iwasawa algebras without any distractions. 

I would also like to thank: James Zhang for the invitation to spend two weeks in Seattle; the ICMS and the University of Washington for hosting conferences during which large parts of this paper were written; Simon Wadsley for his continued interest in my work and many valuable discussions; and Jon Nelson and Rishi Vyas for finding several inaccuracies in an earlier version of this paper.
\section{Preliminaries}

\subsection{Filtered rings}
\label{Filt}
Recall that a \emph{filtration} on a ring $A$ is a function
\[v : A \to \mathbb{R} \cup \{\infty\}:= \mathbb{R}_\infty,\]
such that $v(ab) \geq v(a)+v(b)$, $v(a+b)\geq$ min$\{v(a),v(b)\}$ for all $a,b \in A$, $v(1) = 0$ and $v(0) = \infty$. If the filtration on $A$ is understood, then we say that $A$ is a \emph{filtered ring}. If the stronger condition $v(ab) = v(a) + v(b)$ is satisfied for all $a,b\in A$, then we say that $v$ is a \emph{valuation}. 

We now fix a filtration $v$ on $A$ and define an additive subgroup $A_\lambda$ of $A$ for any $\lambda \in \mathbb{R}$ as follows:
\[A_\lambda := \{x \in A : v(x) \geq \lambda\}.\]
These subgroups have the following properties:
\begin{itemize}
\item $A_\lambda \cdot A_\mu \subseteq A_{\lambda + \mu}$ for all $\lambda,\mu\in\mathbb{R}$,
\item $A_\lambda \supseteq A_\mu$ if $\lambda \leq \mu$,
\item $\cup_{\lambda\in\mathbb{R}} A_\lambda = A$, and
\item $1 \in A_0$.
\end{itemize}
The filtration $v$ is said to be \emph{separated} if the two-sided ideal $v^{-1}(\infty) = \cap_{\lambda\in\mathbb{R}} A_\lambda$ is zero. Since this ideal is proper, we see that any filtration on a field is necessarily separated.

For any $\lambda \in \mathbb{R}$, let $A_{\lambda^+} := \{x \in A : v(x) > \lambda\}$, and define
\[\gr_\lambda A := A_\lambda / A_{\lambda^+}\].
Since $A_{\lambda^+}\cdot A_{\mu} + A_{\lambda}\cdot A_{\mu^+} \subseteq A_{(\lambda + \mu)^+}$ for all $\lambda,\mu\in\mathbb{R}$, the direct sum
\[\gr A := \bigoplus_{\lambda\in\mathbb{R}} \gr_\lambda A\]
is naturally an $\mathbb{R}$-graded ring, called the \emph{associated graded ring} of $A$. The filtration $v$ is a valuation if and only if $\gr A$ is a domain.

\subsection{Zariskian filtrations}
Let $A$ be a filtered ring with filtration $w : A \to \mathbb{R}_\infty$. We say that $w$ is a \emph{Zariskian filtration} if
\begin{itemize}
\item $w$ takes integral values,
\item the Rees ring $\widetilde{A} = \bigoplus_{n \in \mathbb{Z}} A_n t^n \subseteq A[t,t^{-1}]$ is Noetherian,
\item the Jacobson radical of the subring $A_0$ contains $A_1$:
\[1 + A_1 \subseteq A_0^\times.\]
\end{itemize}
This agrees with the standard definition given in \cite{LVO}, except that our filtrations are descending and those in \cite{LVO} are ascending. 

\subsection{Filtered modules}

Let $A$ be a filtered ring with filtration $v$ and let $M$ be a (left) $A$-module. Then a \emph{filtration} on $M$ is a function
\[v_M : M \to \mathbb{R}_\infty,\]
such that $v_M(am) \geq  v(a) + v_M(m)$ and $v_M(m + n) \geq \min\{v_M(m),v_M(n)\}$ for all $m,n \in M$ and $a\in A$. If the filtration on $v_M$ is understood, we will say that $M$ is a \emph{filtered $A$-module}.

The filtration $v_M$ gives rise to a \emph{filtration topology} on $M$, such that $M$ is a topological group under addition, and such that the subgroups 
\[M_{\lambda} := \{m \in M : v_M(m)\geq \lambda\} \]
form a base for the open neighbourhoods of $0$. This topology is Hausdorff if and only if $v_M^{-1}(\infty) = 0$; in this case the filtration $v_M$ is \emph{separated}. We say that $v_M$ is \emph{complete} if every Cauchy sequence in $M$ with respect to this topology converges to a unique limit in $M$. Thus every complete filtration is by definition separated.

\subsection{Bounded linear maps}
\label{BddHoms}
Let $k$ be a field equipped with the trivial valuation $v(k^\times) = 0$, and let $M$ and $N$ be two separated filtered $k$-vector spaces. We say that a $k$-linear map $f : M \to N$ is \emph{bounded} if there exists $\lambda \in \mathbb{R}_\infty$ such that 
\[v_N(f(x)) \geq v_M(x) + \lambda \quad \mbox{for all} \quad  x \in M.\]
The set $\mathcal{B}(M,N)$ of all such maps is a $k$-vector space. The \emph{degree} of a bounded $k$-linear map $f$ is given by
\[\deg(f) := \inf\{v_N(f(x)) - v_M(x) : x \in M \backslash \{0\}\}.\]
The degree function turns $\mathcal{B}(M,N)$ into a separated filtered $k$-vector space and can be viewed as a generalization of the operator norm from functional analysis. In that setting, our next result is well-known --- see, for example \cite[Chapter I, Proposition 3.3]{SchNFA}. We give the proof for the convenience of the reader.

\begin{lem} Let $M$ and $N$ be separated filtered $k$-vector spaces, and suppose that $N$ is complete. Then $\mathcal{B}(M,N)$ is also complete with respect to the degree filtration.
\end{lem}
\begin{proof} Let $(f_n)_n$ be a Cauchy sequence in $\mathcal{B}(M,N)$. For each $x \in M$, the sequence $\left(f_n(x)\right)_n$ is Cauchy, hence converges to an element $f(x) \in N$ because $N$ is complete. The function $x \mapsto f(x)$ is clearly $k$-linear. Now if $\deg f_n \to \infty$ then $f_n \to 0$ by definition, so assume that the sequence $(\deg f_n)_n$ is bounded. It is then eventually constant with value $d$ say. Because each $f_n$ is bounded,
\[v_N(f(x)) = v_N(\lim_{n\to\infty} f_n(x)) = \lim_{n\to\infty}v_N(f_n(x)) \geq \lim_{n\to \infty} \deg f_n + v_M(x) = d + v_M(x)\]
for any non-zero $x \in M$, so $f$ is also bounded. It remains to show that $f_n \to f$.

Fix a non-zero element $x \in M$. Since $(f_n)_n$ is Cauchy, for any $\lambda \in \mathbb{R}$ there exists an integer $t$, independent of $x$, such that $v_N(f_n(x) - f_m(x)) \geq \lambda + v_M(x)$ for all $n,m \geq t$. Since $f_m(x) \to f(x)$ as $m \to \infty$, we can find an integer $m \geq t$ such that $v_N(f_m(x) - f(x)) \geq \lambda + v_M(x)$. Hence
\[ v_N(f_n(x) - f(x)) \geq \min \left\{ v_N(f_n(x) - f_m(x)), v_N(f_m(x) - f(x))\right\} \geq \lambda + v_M(x)\]
for any non-zero $x \in M$, so $f_n \to f$ with respect to the degree filtration.
\end{proof}
Whenever $A$ is a filtered $k$-algebra and $N$ is a filtered $A$-module, $\mathcal{B}(M,N)$ becomes a filtered $A$-module, with the action of $A$ given by $(a.f)(m) = a.f(m)$ for all $a\in A, m\in M$. Note also that $\mathcal{B}(A) := \mathcal{B}(A,A)$ is a filtered ring and $\mathcal{B}(A,N)$ is a filtered right $\mathcal{B}(A)$-module, via composition of functions.

\section{The construction of a non-commutative valuation}
\label{ConstNCVal}
We now start working towards the proof of Theorem \ref{MainC}, which is concluded in $\S$\ref{PfMainZarThm}. Assume from now on that $R$ satisfies the hypotheses of the Theorem.

\subsection{Minimal prime ideals of $\gr R$}\label{MinPrimes}
Because $\gr R$ is a commutative Noetherian $\mathbb{Z}$-graded ring by assumption, it has finitely many minimal primes $\mathfrak{p}_1,\ldots,\mathfrak{p}_m$ say. It is well-known that these ideals are graded.

\begin{lem} At least one of the $\mathfrak{p}_i$ differs from $\gr R$ in at least one non-zero homogeneous component.
\end{lem}
\begin{proof} Suppose not. Let $\mathfrak{n} = \mathfrak{p}_1\cap\cdots\cap\mathfrak{p}_m$ be the prime radical of $\gr R$, a graded ideal. Because $(\gr R) / \mathfrak{n}$ embeds into the direct sum of all the $(\gr R) / \mathfrak{p}_i$, the graded module $(\gr R)/\mathfrak{n}$ is concentrated in degree zero. But $(\gr R)_0 = F$ is a field by assumption, so $\mathfrak{n} \cap (\gr R)_0 = 0$ and therefore $\gr R = F \oplus \mathfrak{n}$. In particular, the factor ring $(\gr R) / \mathfrak{n}$ is isomorphic to $F$. Since $\gr R$ is Noetherian, $\mathfrak{n}^k = 0$ for some $k$ and $\mathfrak{n}^r / \mathfrak{n}^{r+1}$ is a finitely generated $(\gr R)/\mathfrak{n}$-module for all $r \geq 0$. Therefore $\gr R$ must be finite dimensional over $F$, contradicting our assumption on $\gr R$.
\end{proof}

\subsection{Homogeneous localisation}\label{HomogLoc}
We now fix a minimal prime ideal, $\mathfrak{p}_1$ say, which differs from $\gr R$ in at least one non-zero homogeneous component, and define
\[ T := \{X \in \gr R \backslash \mathfrak{p}_1 : X \quad \mbox{is homogeneous}\}.\]
This is a homogeneous multiplicatively closed set, so the localisation $(\gr R)_T$ is a $\mathbb{Z}$-graded ring.

\begin{prop} Let $B := (\gr R)_T = \oplus_{n\in\mathbb{Z}} B_n$ and let $B_{\geq 0} := \oplus_{n \geq 0}B_n$ be the non-negative part of $B$.
\begin{enumerate}[{(}a{)}] \item The ideal $\mathfrak{p} := (\mathfrak{p}_1)_T$ of $B$ is nilpotent.
\item $B_0$ is a local ring with maximal ideal $B_0 \cap \mathfrak{p}$.
\item There exists a homogeneous element $Y \in B$ of positive degree such that 
\[B / \mathfrak{p} \cong (B/\mathfrak{p})_0[\overline{Y},\overline{Y}^{-1}].\]
\item $B$ is a finitely generated $B_0[Y,Y^{-1}]$-module and $B_0$ is Artinian.
\item $B$ is $\gr$-Artinian: every descending chain of graded ideals terminates.
\item $B_{\geq 0} / Y B_{\geq 0}$ is an Artinian ring.
\end{enumerate}\end{prop}
\begin{proof}(a) Because $\gr R$ is Noetherian, some product of the minimal primes of $\gr R$ is zero:
\[ \mathfrak{p}_1^{n_1} \cdot \mathfrak{p}_2^{n_2} \cdot \hspace{1mm} \cdots \hspace{1mm} \cdot \mathfrak{p}_m^{n_m} = 0.\]
If $\mathfrak{p}_2^{n_2} \cdot \cdots \cdot \mathfrak{p}_m^{n_m} \subseteq \mathfrak{p}_1$ then $\mathfrak{p}_i \subseteq \mathfrak{p}_1$ for some $i>1$, which forces $\mathfrak{p}_i = \mathfrak{p}_1$ because $\mathfrak{p}_1$ is a minimal prime. But the $\mathfrak{p}_i$ are all distinct, so
\[\mathfrak{p}_2^{n_2} \cdot \hspace{1mm} \cdots \hspace{1mm} \cdot \mathfrak{p}_m^{n_m} \nsubseteq \mathfrak{p}_1\]
and we can find some homogeneous element $t \in \mathfrak{p}_2^{n_2} \cdot \cdots \cdot \mathfrak{p}_m^{n_m} \backslash \mathfrak{p}_1$. Hence
$\mathfrak{p_1}^{n_1} t = 0$ and therefore $\mathfrak{p}$ is nilpotent.

(b) Let $r/t \in B_0 \backslash \mathfrak{p}$ for some $r \in \gr R$ and $t \in T$. Then $r \in B_0 t$ is homogeneous because $t$ is homogeneous, and $r \notin \mathfrak{p}_1$. So $r \in T$ and $r/t$ is a unit in $B$. Because $r$ and $t$ have the same degree, the inverse $t/r$ lies in $B_0$, so every element of $B_0 \backslash B_0 \cap \mathfrak{p}$ is a unit in $B_0$.

(c) The argument used in part (b) above shows that $D:=B/\mathfrak{p}  = \oplus_{n\in\mathbb{Z}}D_n$ is a \emph{$\gr$-field}: every non-zero homogeneous element of $D$ is a unit. Moreover, $D_n \neq 0$ for some non-zero $n$ by construction. Therefore the set $\{n \in \mathbb{Z} : D_n \neq 0\}$ is a non-zero subgroup of $\mathbb{Z}$, and hence equals $\ell \mathbb{Z}$ for some $\ell > 0$. Pick $Y \in B_\ell$ whose image $\overline{Y}$ in $D$ is non-zero. Now if $x \in D_{\ell k}$, then $x\overline{Y}^{-k} \in D_0$, so
\[D = D_0[\overline{Y},\overline{Y}^{-1}].\]
Because $\mathfrak{p}$ is a graded ideal of $B$, $D_0 = (B/\mathfrak{p})_0 = B_0/\mathfrak{p}_0 = B_0 / (B_0 \cap \mathfrak{p})$ is the residue field of $B_0$.

(d) By part (a), we have a finite filtration of $B$ by graded ideals:
\[B > \mathfrak{p} > \mathfrak{p}^2 > \cdots > \mathfrak{p}^{n_1} = 0.\]
Each subquotient $\mathfrak{p}^i / \mathfrak{p}^{i+1}$ is finitely generated over $B / \mathfrak{p} \cong D_0[\overline{Y}, \overline{Y}^{-1}]$, so $B$ is finitely generated over $B_0[Y,Y^{-1}]$. Also, $(\mathfrak{p}^i/\mathfrak{p}^{i+1})_0$ is finite dimensional over $D_0$ for all $i \geq 0$, so $B_0$ admits a finite filtration
\[ B_0 > \mathfrak{p}_0 > (\mathfrak{p}^2)_0 > \cdots > (\mathfrak{p}^{n_1})_0 = 0\]
with each subquotient finite dimensional over the residue field $D_0$. Hence $B_0$ is Artinian.

(e) By part (a), we can find a finite composition series consisting of graded ideals for $B$, with each factor isomorphic to $B/\mathfrak{p}$ (possibly with shifted degrees). But $B/\mathfrak{p}$ has no proper non-zero graded ideals because it is a $\gr$-field. So $B$ is $\gr$-Artinian.

(f) Let $x_1,\ldots,x_r$ be a generating set for $B$ as a $B_0[Y,Y^{-1}]$-module consisting of homogeneous elements. By multiplying these generators by a power of $Y$, we may assume that $d_i := \deg(x_i) > 0$ for all $i$. Hence
\[B_k = \sum_{i=1}^r x_i B_0 Y^{\frac{k - d_i}{\ell}}\]
for all $k \in \mathbb{Z}$, with the understanding that fractional powers of $Y$ are zero. Now $B_k \subseteq Y B_{\geq 0}$ whenever $k > \max d_i + \ell$, so the factor ring $B_{\geq 0} / Y B_{\geq 0}$ is concentrated in finitely many degrees and is therefore a finitely generated $B_0$-module. Since $B_0$ is Artinian by part (d), this ring must also be Artinian.
\end{proof}

\subsection{Ore localisation}
\label{OreLoc}
The \emph{saturated lift} of $T$, namely
\[ S := \{r \in R : \gr(r) \in T\}\]
is a right and left Ore set in $R$ by \cite[Corollary 2.2]{Li}. By \cite[Proposition 2.3]{Li}, the Ore localisation $R_S$ carries a filtration such that
\[ \gr (R_S) \cong (\gr R)_T\]
and this filtration is actually Zariskian by \cite[Proposition 2.8]{Li}.
\begin{lem}$R_S$ is equal to the classical ring of quotients $Q$ of $R$.
\end{lem}
\begin{proof} Because the filtration on $R_S$ is Zariskian and because $\gr(R_S)$ is $\gr-$Artinian by Proposition \ref{HomogLoc}(e), it follows from \cite[Chapter II, Corollary 3.1.2]{LVO} that $R_S$ is an Artinian ring. Now $R$ is prime by assumption and $0 \notin S$ because $0 \notin T$. Let $\ass(S)$ be the right $S$-torsion submodule of $R$; then $\ass(S)$ is a two-sided ideal of $R$ by \cite[Lemma 2.1.9]{MCR} and is finitely generated as a left ideal since $R$ is left Noetherian. Let $\ass(S) = \sum_{i=1}^tRx_i$ and choose $s_i \in S$ such that $x_is_i = 0$. Since $S$ is a right Ore set, by \cite[Lemma 2.1.8]{MCR} we can find elements $s' \in S$ and $r_i' \in R$ such that $s' = s_ir_i'$ for all $i$. Now 
\[ \ass(S) \cdot (Rs'R) = \ass(S)s'R = \sum_{i=1}^t Rx_i s_i r_i'R = 0\]
so $\ass(S) = 0$ because $R$ is prime and $s'\neq 0$. Hence $S$ consists of regular elements in $R$ by \cite[Proposition 2.1.10(ii)]{MCR}, even though $T$ may contain zero-divisors in $\gr R$. Therefore $R_S$ is a subring of $Q$ containing $R$. But $R_S$ is Artinian and regular elements in $R$ stay regular in $R_S$, so every regular element of $R$ is a unit in $R_S$ by \cite[Proposition 3.1.1]{MCR}. Therefore $R_S = Q$ as claimed.\end{proof}

\subsection{Microlocalisation}
\label{MicroLoc}By definition, the \emph{microlocalisation} of $R$ at the homogeneous set $T$ is the completion $\widehat{Q}$ of $R_S$ with respect to the Zariskian filtration used in $\S$\ref{OreLoc}. This ring still carries a natural Zariskian filtration $\deg$, with respect to which we have
\[ \gr \widehat{Q} \cong \gr(R_S) \cong (\gr R)_T.\]
So $\widehat{Q}$ is still Artinian. However, in general it will not be a simple ring; worse still, it may fail to be semi-simple even in the case when $R$ is a commutative domain. We will deal with this issue very soon, but let us first focus on the ``unit ball" of $\widehat{Q}$, namely
\[U := (\widehat{Q})_0 = \{u \in \widehat{Q} : \deg(u) \geq 0\}.\]
It follows from \cite[Chapter II, Lemma 2.1.4]{LVO} that $U$ is Noetherian. Proposition \ref{HomogLoc} now translates into the following properties of $U$:

\begin{prop} There exists a regular normal element $y \in J(U)$ such that $U / yU$ is Artinian and $U$ is $y$-adically complete. $U$ has Krull dimension at most $1$ on both sides and $U/J(U)$ is a commutative field. \end{prop}
\begin{proof} Equip $U \subseteq \widehat{Q}$ with the subspace filtration and identify $\gr U$ with the positive part $(\gr \widehat{Q})_{\geq 0}$ of $\gr \widehat{Q} \cong (\gr R)_T$. Choose a homogeneous element $Y \in \gr \widehat{Q}$ as in Proposition \ref{HomogLoc} and choose any lift $y \in U$ such that $\gr y = Y$. Since $Y$ has positive degree, $y \in U \cap (\widehat{Q})_1 = U_1 \subseteq J(U)$ because the filtration on $\widehat{Q}$ is Zariskian.

Now $Y$ is a homogeneous unit in $\gr \widehat{Q}$ because it is a unit in $\gr \widehat{Q} / \mathfrak{p}$ by construction and $\mathfrak{p}$ is nilpotent by Proposition \ref{HomogLoc}(a). Therefore $y$ is a unit in $\widehat{Q}$ and
\[ \deg( y^{-1} u y ) = \deg(u)\]
for all $u \in \widehat{Q}$. So $y^{-1}Uy = U$ and hence $y \in U$ is a regular normal element. Now
\[ \gr(U/yU) = (\gr U)/(Y \gr U)\]
is Artinian by Proposition \ref{HomogLoc}(f), so $U /yU$ must also be Artinian.

To show that $U$ is $y$-adically complete, it is sufficient to show that the $y$-adic filtration is topologically equivalent to the degree filtration on $U$, since the latter is complete by definition of $\widehat{Q}$. Since $\deg(y) \geq 1$, $U_n$ contains $y^nU$. Given $y^n U$ choose an integer $m$ larger than $n \deg (y)$ and let $u \in U_m$; then $\deg(y^{-n}u) = \deg(u) - n \deg(y) \geq m - n\deg(y) \geq 0$ so $y^{-n}U_m \subseteq U$ and $y^nU$ contains $U_m$. 

Because $y \in U$ is normal, the associated graded ring of $A$ with respect to the $y$-adic filtration is a skew polynomial ring
\[\gr_y U = (U/yU)[X; \sigma]\]
where $\sigma$ is the ring automorphism of $U/yU$ induced by conjugation by $y$. Since $U/yU$ is Artinian, $\mathcal{K}(\gr_y U) \leq 1$ by \cite[Proposition 6.5.4(i)]{MCR}. Since the $y$-adic filtration on $U$ is complete, $\mathcal{K}(U) \leq \mathcal{K}(\gr_y U) \leq 1$ by \cite[Chapter II, Corollary 3.1.2(2)]{LVO}.

Now $U/U_1 \cong ((\gr R)_T)_0$ is a commutative local Artinian ring by Proposition \ref{HomogLoc}. Let $\mathfrak{m} / U_1$ be its maximal ideal; then $U / \mathfrak{m}$ is a commutative field and $\mathfrak{m}^n \subseteq U_1 \subseteq J(U)$ for some $n$, because the filtration on $\widehat{Q}$ is Zariskian. Hence $J(U) = \mathfrak{m}$.\end{proof}

\subsection{Prime factor rings of $U$}\label{AProps} Let $A$ be the factor ring of $U$ by any of its minimal prime ideals. Then $A$ is a prime Noetherian ring and we will denote its classical ring of quotients by $Q(A)$. This is one of the finitely many prime factor rings of $\widehat{Q}$.

\begin{prop} There exists a regular normal element $z \in J(A)$ such that $A/zA$ is Artinian and $A$ is $z$-adically complete. The ring $A$ has Krull dimension at most $1$ on both sides and $A/J(A)$ is a commutative field. Moreover $Q(A)$ is the localisation $A_z$ of $A$ at the powers of the regular normal element $z \in A$. \end{prop}
\begin{proof} Let $z \in A$ be the image of the element $y \in U$ given by Proposition \ref{MicroLoc}. This element is normal; it is non-zero because the map $U \to Q(A)$ factors through $\widehat{Q}$ and because $y$ is a unit in $\widehat{Q}$. Non-zero normal elements in a prime ring are necessarily regular. Proposition \ref{MicroLoc} also implies that $\mathcal{K}(A) \leq 1$ and that $A/J(A)$ is a commutative field. Since $U$ is $y$-adically complete and $\gr_y U$ is Noetherian, every ideal of $U$ is closed in the $y$-adic topology by \cite[Chapter II, Theorem 2.1.2]{LVO}, so $A$ is also $z$-adically complete.

Since $z \in A$ is regular and normal, its powers form an Ore set in $A$ and we can form the partial localisation $A_z \subseteq Q(A)$. Now if $ac^{-1} \in Q(A)$ for some $a \in A$ and some regular element $c \in A$, then the descending chain of right ideals $A > cA > c^2A > \cdots $ has each subquotient isomorphic to $A/cA$, so $A/cA$ must have finite length as an $A$-module since $\mathcal{K}(A) \leq 1$. Hence $z^t \in cA$ for some $t$ because $z\in J(A)$, so $z^t = cx$ for some $x \in A$. Therefore $ac^{-1} = axz^{-t} \in A_z$ and $Q(A) = A_z$.
\end{proof}

\subsection{Orders and maximal orders}\label{StructThm}
Let $B$ be a subring of $Q(A)$ containing $A$. Recall \cite[\S 3.1.9]{MCR} that $B$ is \emph{equivalent} to $A$ as an order if there are units $a, b \in Q(A)$ such that $aBb \subseteq A$. We define
\[\mathcal{S} := \{ B \leq Q(A) : A \leq B \quad\mbox{and}\quad B\quad\mbox{is equivalent to}\quad A\}.\]
Elements of $\mathcal{S}$ maximal with respect to inclusion are called \emph{maximal orders}. It turns out that these maximal orders have a very precise structure. We say that a Noetherian domain $D$ with skewfield of fractions $Q(D)$ is a \emph{non-commutative discrete valuation ring} if for all non-zero $x \in Q(D)$, either $x \in D$ or $x^{-1} \in D$.
\begin{thm} Let $B \in \mathcal{S}$ be a maximal order. Then there exists an integer $k \geq 1$ and a non-commutative complete discrete valuation ring $D$, such that $B$ is isomorphic to a complete $k \times k$ matrix ring over $D$:
\[B \cong M_k(D).\]
\end{thm}
The proof is given below in $\S$\ref{PfStThm}.

\subsection{Rings of Krull dimension $1$}\label{Kdim1}
Orders $B \in \mathcal{S}$ have properties similar to $A$. More precisely:
\begin{prop}Let $B\in\mathcal{S}$. Then
\begin{enumerate}[{(}a{)}]
\item $B$ is contained in $Az^{-k} = z^{-k}A$ for some $k \geq 0$,
\item $B$ is a prime, Noetherian order in $Q(A)$,
\item $B$ has right and left Krull dimension 1,
\item $B$ is semilocal,
\item $B$ is right and left bounded.
\end{enumerate}
\end{prop}
\begin{proof}
(a) Since $B$ is equivalent to $A$ we can find units $a,b \in Q(A)$ such that $B \subseteq a^{-1} A b^{-1}$. By Proposition \ref{AProps}, $Q(A) = A_z$ so there exists an integer $k$ such that $a^{-1}, b^{-1} \in Az^{-k}$. Since $z \in A$ is normal, $B \subseteq Az^{-k} = z^{-k}A$.

(b) By part (a), $B$ is a Noetherian $A$-module on both sides, so $B$ is itself Noetherian. Also $B$ is a prime order in $Q(A)$ by \cite[Corollary 3.1.6 (i)]{MCR}.

(c) Since $B$ is finitely generated over $A$ on both sides, \cite[Lemma 6.2.5]{MCR} and Proposition \ref{AProps} together imply
\[\mathcal{K}(B_B) \leq \mathcal{K}(B_A) \leq \mathcal{K}(A_A) \leq 1.\]
If $\mathcal{K}(B) = 0$, then the regular element $z \in B$ is a unit in $B$ by \cite[Proposition 3.1.1]{MCR}, and $A < z^{-1}A < z^{-2}A < z^{-3}A < \cdots$ is a strictly ascending chain in the Noetherian $A$-module $B$. Hence $\mathcal{K}(B_B) = 1$ and similarly $\mathcal{K}(_BB) = 1$.

(d) Since $z \in A$ is normal, $B/Bz$ is an $A-A/zA$-bimodule, which is finitely generated on both sides. Since $A/zA$ is Artinian, $B/Bz$ must also be Artinian as a left $A$-module by Lenagan's Lemma \cite[Theorem 4.1.6]{MCR}. Because $z \in J(A)$, we deduce that $z^nB \subseteq Bz$ for some $n \geq 1$. Now if $M$ is a simple right $B$-module, then $M$ is a finitely generated non-zero right $A$-module, so $Mz < M$ by Nakayama's Lemma. Hence
\[ M z^nB \subseteq  MBz \subseteq Mz < M;\]
but $M z^nB$ is a $B$-submodule of $M$, so $Mz^nB = 0$ because $M$ is simple. Therefore $z^n \in J(B)$ and hence $B/J(B)$ is Artinian as a right $A$-module and therefore as a right $B$-module. Hence $B$ is semilocal.

(e) Recall that $B$ is \emph{right bounded} if every essential right ideal of $B$ contains a non-zero two-sided ideal of $B$. Now if $I$ is an essential right ideal of $B$, then $\mathcal{K}(B/I) < \mathcal{K}(B) = 1$ by \cite[Proposition 6.3.10(i)]{MCR}, so $B/I$ is Artinian and therefore $J(B)^m \subseteq I$ for some $m$. Now $J(B) \neq 0$ because $\mathcal{K}(B) = 1$ by part (c) and $\mathcal{K}(B/J(B)) = 0$ by part (d). Finally $B$ is prime, so $J(B)^m$ is a non-zero two-sided ideal of $B$ contained in $I$. A similar argument shows that $B$ is also left bounded.
\end{proof}

\subsection{Reflexive ideals}\label{MaxRefl}
Recall that an essential left ideal $I$ of an order $B \in \mathcal{S}$ is \emph{reflexive} if $(I^{-1})^{-1} = I$ where $I^{-1} = \{q \in Q(A) : Iq \subseteq B\}$ and $(I^{-1})^{-1} = \{q \in Q(A) : qI^{-1} \subseteq B\}$. Reflexive right ideals are defined similarly. A \emph{prime $c$-ideal} is a non-zero prime ideal of $B$ which is reflexive as a left ideal. In the case when $B$ is a maximal order, it follows from \cite[Proposition 5.1.8]{MCR} that a prime ideal is reflexive as a left ideal if and only if it is reflexive as a right ideal.

\begin{prop} Let $B \in\mathcal{S}$ be a maximal order. Then every non-zero prime ideal $I$ of $B$ is reflexive.
\end{prop}
\begin{proof} Since $\mathcal{K}(B) = 1$ by Lemma \ref{Kdim1}(c), $B/I$ is Artinian by \cite[Proposition 6.3.11(ii)]{MCR} so $I$ contains a power of $J(B)$. Because $I$ is prime, $J(B) \subseteq I$. Since $B$ is semilocal by Lemma \ref{Kdim1}(d), we see that $B$ only has finitely many non-zero prime $c$-ideals $P_1,\ldots, P_n$, say.

Since $B$ is a prime maximal order, every prime $c$-ideal of $B$ is localisable by a result of Goldie \cite{Goldie1967} --- see also \cite[Proposition 1.7]{Cham1981}. Let $B_i$ denote the localisation of $B$ at $\mathcal{C}(P_i)$. Because $B$ is bounded by Lemma \ref{Kdim1}(e), it follows from the work of Chamarie \cite[Proposition 1.10(b)]{Cham1981} that
\[B = B_1 \cap B_2 \cap \cdots \cap B_n.\]
Note that this result implies that $B$ has at least one prime $c$-ideal. Moreover each $B_i$ is a local ring, with Jacobson radical $J(B_i) = P_iB_i$, by \cite[Proposition 1.9]{Cham1981}. 

Let $x \in P_1 \cap P_2 \cap \cdots \cap P_n$. Then $x \in J(B_i)$ for all $i$, so
\[1 + BxB \subseteq 1 + B_ixB_i \subseteq B_i^\times\]
for all $i$, and therefore $1 + BxB \subseteq B_1^\times \cap \cdots \cap B_n^\times \subseteq B^\times$. Hence $x \in J(B)$:
\[P_1\cdot P_2 \cdot \hspace{1mm} \cdots \hspace{1mm}\cdot P_n \subseteq  P_1 \cap P_2 \cap \cdots \cap P_n \subseteq J(B) \subseteq I.\]
Because $I$ is prime, we deduce that $P_i \subseteq I$ for some $i$. But $P_i$ is a maximal two-sided ideal since $J(B) \subseteq P_i$. Hence $I = P_i$ is reflexive.
\end{proof}

\subsection{Dedekind prime rings}\label{Dedekind} Recall that a Noetherian ring $B$ is \emph{left (right) hereditary} if every left (right) ideal of $B$ is projective. Equivalently, $B$ has left (right) global dimension $\leq 1$. $B$ is said to be a \emph{Dedekind prime ring} if
\begin{itemize}
\item $B$ is a prime maximal order,
\item $B$ is left and right hereditary.
\end{itemize}

\begin{prop} Let $B \in \mathcal{S}$ be a maximal order. Then $B$ is a Dedekind prime ring.
\end{prop}
\begin{proof} By symmetry, it is enough to show that $B$ is left hereditary. Let $P$ be a maximal two-sided ideal of $B$. Then $P$ is reflexive by Proposition \ref{MaxRefl}. Now $P \subseteq P^{-1}P \subseteq B$ so $P^{-1}P = P$ or $P^{-1}P = B$ by the maximality of $P$. But $P^{-1}P = P$ implies $P^{-1} \subseteq \mathcal{O}_l(P) = B$ because $B$ is a maximal order, and then $B = (P^{-1})^{-1} = P$, a contradiction. So $P^{-1}P = B$: every maximal two-sided ideal of $B$ is left invertible. It now follows from the Dual Basis Lemma \cite[Lemma 3.5.2(ii)]{MCR} that every maximal two-sided ideal of $B$ is projective as a left $B$-module.

Let $M$ be a simple left $B$-module. Since $B$ is semilocal by Lemma \ref{Kdim1}(d), $P := \Ann_B(M)$ is a maximal ideal of $B$ and $B / P$ is isomorphic to a direct sum of finitely many copies of $M$ as a left $B$-module. Let $\pd(N)$ denote the projective dimension of a $B$-module $N$; then
\[ \pd(M) = \pd(B/P) \leq 1\]
because $P$ is projective.  Now let $I$ be any non-zero left ideal of $B$; then we can find another left ideal $J$ of $B$ such that $L := I \oplus J$ is essential. Since $\mathcal{K}(B) \leq 1$, $B / L$ has finite length by \cite[Proposition 6.3.10(i)]{MCR} and therefore
\[\pd(B/L) \leq 1\]
by \cite[\S 7.1.6]{MCR}. Hence $L$ is projective by Schanuel's Lemma \cite[\S 7.1.2]{MCR} and therefore $I$ is also projective.
\end{proof}

\subsection{Proof of Theorem \ref{StructThm}}\label{PfStThm}
The hard work has already been done; it remains to apply a result of Gwynne and Robson \cite{GwyRob}.

By Proposition \ref{Dedekind}, $B$ is a Dedekind prime ring. So $B$ is an Asano order, by \cite[Theorem 5.2.10]{MCR}. Let $P_1,\ldots, P_n$ be the maximal two-sided ideals of $B$ and $J = J(B)$. Then $J = P_1P_2 \cdots P_n$ and
\[J^k = P_1^k P_2^k \cdots P_n^k = P_1^k \cap P_2^k \cap \cdots \cap P_n^k \quad\mbox{for all} \quad k \geq 1\]
by \cite[Theorem 5.2.9]{MCR}. Hence each factor ring $B/J^k$ decomposes as a direct sum
\[\frac{B}{J^k} \cong \frac{B}{P_1^k} \oplus \frac{B}{P_2^k} \oplus \cdots \oplus \frac{B}{P_n^k}.\]
Passing to the inverse limit, we see that the $J$-adic completion of $B$ isomorphic to the direct sum of the $P_i$-adic completions of $B$:
\[\widehat{B}^J \cong \widehat{B}^{P_1} \oplus \widehat{B}^{P_2} \oplus \cdots \oplus \widehat{B}^{P_n}.\]
But $B$ is a finitely generated $A$-module which is $z$-adically complete by Proposition \ref{AProps}, so $B$ is also complete with respect to the $z$-adic filtration
\[B > Bz > Bz^2 > \cdots \,.\]
We saw in the proof of Lemma \ref{Kdim1}(d) that $z^n \in J(B)$ for some $n$ and that $B/Bz$ is an Artinian left $A$-module; therefore $B/Bz$ is also an Artinian left $B$-module and $J(B)^m \subseteq Bz$ for some $z$. It follows that $B$ is $J(B)$-adically complete: $\widehat{B}^J = B$.

Since $B$ is prime, we deduce that $B$ has a unique non-zero prime ideal $P$ and $\widehat{B}^P = B$. In this situation, \cite[Theorem 2.3]{GwyRob} states that 
\[B \cong M_k(D)\]
for some complete, scalar local, principal ideal domain $D$. But any such $D$ is a non-commutative complete discrete valuation ring. \qed

\subsection{Existence of maximal orders in $\mathcal{S}$}
\label{ExistsMO}
The theory developed above must be well-known to the experts. However it would not be very useful unless we could show that maximal orders in $\mathcal{S}$ actually \emph{exist}. Our assumptions on $A$ are fortunately strong enough to allow us to prove precisely this. 
\begin{defn} The \emph{left conductor} of $B\in \mathcal{S}$ is the largest left ideal $I_B$ of $B$ contained in $A$.\end{defn}

\begin{lem} $I_B$ is a non-zero two-sided ideal of $A$, and $I_C \subseteq I_B$ whenever $B \subseteq C$ are in $\mathcal{S}$.
\end{lem}
\begin{proof} By Proposition \ref{Kdim1}(a), $B$ is contained in $Az^{-k}$ for some $k \geq 0$. Hence $Bz^k$ is a non-zero left ideal of $B$ contained in $A$, whence $I_B \neq 0$. If $a \in A$ then $I_Ba \subseteq A$ is still a left ideal of $B$ contained in $A$ so $I_Ba \subseteq I_B$ by the maximality of $I_B$; hence $I_B$ is a two-sided ideal of $A$. Finally if $B \subseteq C$ then $B I_C \subseteq I_C$ so $I_C$ is a left ideal of $B$ contained in $A$, whence $I_C \subseteq I_B$.
\end{proof}
We need one more preparatory result.
\begin{prop}Suppose that $I_1 \supseteq I_2 \supseteq I_3 \supseteq \cdots $ is a descending chain of left ideals of $A$ such that $I_n \nsubseteq Az$ for all $n$. Then $I_\infty := \cap_{n=1}^\infty I_n$ is non-zero.
\end{prop}
\begin{proof} Since $A/Az$ is Artinian, the chain
\[ I_1 + Az \supseteq I_2 + Az \supseteq \cdots \]
stops: there exists $k_1 \geq 1$ such that $I_n + Az = I_{k_1} + Az$ for all $n \geq k_1$. Pick $x_1 \in I_{k_1} \backslash Az$. Since $A / Az^2$ is Artinian, the chain
\[ I_1 + Az^2 \supseteq I_2 + Az^2 \supseteq \cdots \]
stops: there exists $k_2 > k_1$ such that $I_n + Az^2 = I_{k_2} + Az^2$ for all $n \geq k_2$. Pick $x_2 \in I_{k_2}$ such that $x_2 \equiv x_1 \mod Az$. Continuing like this, we construct a sequence of integers $1 \leq k_1 < k_2 < k_3 < \cdots$ and a sequence of elements
\[x_1 \in I_{k_1}, \quad x_2 \in I_{k_2}, \quad x_3 \in I_{k_3},\quad \cdots , \]
such that $x_n \equiv x_{n-1} \mod Az^{n-1} $ for all $n$.

Since $A$ is $z$-adically complete by assumption, the limit
\[x_\infty := \lim_{n\to \infty} x_n\]
exists in $A$. Fix $n \geq 1$; then $x_m \in I_{k_m} \subseteq I_m \subseteq I_n$ whenever $m \geq n$ because $k_m \geq m$. Since the $z$-adic filtration on $A$ is Zariskian by Lemma \ref{AProps}, each left ideal $I_n$ is closed by \cite[Chapter II, Theorem 2.1.2]{LVO}, so $x_\infty \in I_n$ for all $n \geq 1$. Moreover $x_\infty \equiv x_1 \mod Az$ so $x_\infty$ is non-zero by construction. Hence
\[0 \neq x_\infty \in \cap_{n=1}^\infty I_n\]
as claimed.
\end{proof}

\begin{thm}
The collection $\mathcal{S}$ of orders containing $A$ and equivalent to $A$ satisfies the ascending chain condition.
\end{thm}
\begin{proof} Let $B_1 \subseteq B_2 \subseteq B_3 \subseteq \cdots$ be an ascending chain in $\mathcal{S}$. Let $I_n = I_{B_n}$ be the left conductor of $B_n$; then
\[ I_1 \supseteq I_2 \supseteq I_3 \supseteq \cdots\]
is a descending chain of non-zero two-sided ideals of $A$ by the Lemma. If $I_n \subseteq Az$ for some $n$ then $I_nz^{-1} \subseteq A$ is still a left ideal of $B_n$ so $I_n z^{-1} \subseteq I_n$ by the maximality of $I_n$. Therefore $I_nz = I_n$, which forces $I_n = 0$ by Nakayama's Lemma, a contradiction --- so in fact $I_n \nsubseteq Az$ for any $n$. Since $A$ is $z$-adically complete, $I_\infty := \cap_{n=1}^\infty I_n$ is non-zero by the Proposition.

Fix $n \geq 1$ and let $m \geq n$. Then
\[B_n I_\infty \subseteq B_n I_m \subseteq B_m I_m \subseteq I_m\]
so $B_n I_\infty \subseteq I_\infty$ for all $n$. Hence every term $B_n$ in our ascending chain is contained in $\mathcal{O}_l(I_\infty) := \{q \in Q(A) : qI_\infty \subseteq I_\infty\}$. Since $I_\infty$ is a non-zero two-sided ideal of the prime ring $A$, it contains a regular element by Goldie's Theorem \cite[Proposition 2.3.5(ii)]{MCR}. Therefore $\mathcal{O}_l(I_\infty) \in \mathcal{S}$ by \cite[Lemma 3.1.12(i)]{MCR}. In particular, $\mathcal{O}_l(I_\infty)$ is a Noetherian $A$-module on both sides by Proposition \ref{Kdim1}(a). So the chain $B_1 \subseteq B_2 \subseteq B_3 \subseteq \cdots \subseteq \mathcal{O}_l(I_\infty)$ of $A$-modules must terminate.
\end{proof}

\subsection{Remarks}
1. Theorem \ref{ExistsMO} and Proposition \ref{ExistsMO} both fail if $A$ is not assumed to be $z$-adically complete, even in the case when $A$ is commutative. This is clearly illustrated by Akizuki's example \cite{Ak} of a one-dimensional commutative Noetherian local domain $A$ whose integral closure is not a finitely generated $A$-module. See \cite{ReidAki} for a more modern version of this example. This explains the need to pass to the microlocalisation of $R$.

2. It is well-known \cite[Th\'eor\`eme 23.1.5]{EGAIV1} that a commutative \emph{complete} local Noetherian domain $A$ is a Japanese ring, so in particular the integral closure of $A$ in its field of fractions is a finitely generated $A$-module. This is usually proved using Cohen's Structure Theorem for complete local commutative Noetherian rings, which is not available in the non-commutative case. In the special case when $\mathcal{K}(A) = 1$, Theorem \ref{ExistsMO} gives another proof of this fact: the maximal order $B$ is a commutative complete discrete valuation ring by Theorem \ref{StructThm} and it is integral over $A$, so it must be the integral closure of $A$ in $Q(A)$.

\subsection{Properties of $B/J(B)$}\label{CSA}Before we can give the proof of Theorem \ref{MainC}, we need to study the factor ring $B/J(B)$ more carefully.

\begin{prop} Let $B \in \mathcal{S}$ be a maximal order. Then $C:=B/J(B)$ is a central simple algebra and the associated graded ring of $B$ with respect to the $J(B)$-adic filtration is isomorphic to the polynomial ring $C[X]$.
\end{prop}
\begin{proof} By Theorem \ref{StructThm}, $B$ is isomorphic to $M_k(D)$ for some non-commutative discrete valuation ring $D$. Pick any element $c \in J(D) \backslash J(D)^2$; then $c$ is a regular normal element in $B$ which generates $J(B)$, and
\[\gr B = C[\gr c; \alpha]\]
is a skew-polynomial ring, where $\alpha : C\to C$ is the automorphism induced by conjugation by $c$.

By Proposition \ref{Kdim1}(a), $C$ is a finitely generated Artinian $A$-module on both sides. By Proposition \ref{AProps}, $A$ is scalar local with maximal ideal $J(A)$ and commutative residue field $A/J(A)$. Hence $C$ is a finitely generated right $A/J(A)^t$ module for some $t \geq 1$, say. Now because $A/J(A)$ is commutative and $J(A)/J(A)^t$ is nilpotent, $A/J(A)^t$ satisfies the polynomial identity $(xy-yx)^t$. Hence $C$ is a PI ring by \cite[Corollary 13.4.9(i)]{MCR}. But $C = B/J(B)$ is primitive by construction, so $C$ is a central simple algebra by Kaplansky's Theorem \cite[Theorem 13.3.8]{MCR}.

Now by the Skolem-Noether Theorem \cite[Theorem 3.1.2]{Rowen}, the automorphism $\alpha : C \to C$ is given by conjugation by some element $q + cB \in C^\times$. Because $cB = J(B)$, $q$ must be a unit in $B$. Replacing the uniformizer $c$ by $q^{-1}c$ then has the effect of making the symbol $X$ of $c$ central in the graded ring $\gr B$, so $\gr B \cong C[X]$ as claimed. \end{proof}

\subsection{Proof of Theorem \ref{MainC}}
\label{PfMainZarThm}
By Theorem \ref{ExistsMO}, we can find a maximal order $B$ of $Q(A)$ equivalent to $A$. Then we have the following commutative diagram of rings, where the vertical maps are inclusions of the rings in the top row into their respective classical rings of quotients:
\[\xymatrix{ R \ar@{^{(}->}[r]\ar[d] &  U \ar@{>>}[r]\ar[d] & A \ar@{^{(}->}[r]\ar[d] & B \ar[d] \\
 Q \ar@{^{(}->}[r] & \widehat{Q} \ar@{>>}[r]_{\eta} & Q(A) \ar@{=}[r] & Q(B).}\]
Let $\eta:\widehat{Q} \twoheadrightarrow Q(A)$ be the natural surjection, and let $v : Q \to \mathbb{Z}\cup \{\infty\}$ be the restriction of the $J(B)$-adic filtration on $Q(A)$ to $Q$; thus
\[ v(x) = \min\{ n \in \mathbb{Z} : \eta(x) \in J(B)^n\}\]
if $x \neq 0$ and $v(0) = \infty$. Note that this filtration is separated because $Q$ is a simple ring and $v(1) = 0$.

(a) Recall from Proposition \ref{HomogLoc}(c) that $Y \in \gr \widehat{Q} \cong (\gr R)_T$ is a unit of degree $\ell > 0$. By construction, the element $y \in J(U)$ satisfies $\gr y = Y$, so
\[ (\widehat{Q})_{k\ell} = y^k U\quad\mbox{for all}\quad k \in\mathbb{Z}.\]
By the proof of Lemma \ref{Kdim1}(d), $z^t \in J(B)$ for a large enough integer $t$, where $z = \eta(y) \in J(A)$. Hence
\[ \eta\left( (\widehat{Q})_{tn\ell}\right) = z^{tn}A \subseteq J(B)^n\]
for all $n \geq 0$. This means that the map $\eta$ is continuous with respect to the natural filtration on $\widehat{Q}$ and the $J(B)$-adic filtration on $Q(A)$. But the map $R \to \widehat{Q}$ is continuous by the definition of the filtration on $Q$ so the composite map $(R,w) \to (Q, v)$ must also be continuous.

(b) This is clear, because $\eta(R_0) \subseteq \eta(U) = A \subseteq B$ by construction.

(c) The inclusion $Q\hookrightarrow \widehat{Q}$ is continuous with dense image, and $\eta : \widehat{Q} \to Q(A)$ is a continuous surjection. Therefore $\eta(Q)$ is dense in $Q(A)$, and 
\[\gr^v Q \cong \gr \eta(Q) = \gr Q(A).\]
By Proposition \ref{CSA}, $\gr B = C[X]$ where $C = B/J(B)$ is a central simple algebra, so $\gr Q(A) = C[X,X^{-1}]$.  

Finally, the restriction of $v$ to $Q(D)$ is a valuation by construction because $D$ is a non-commutative discrete valuation ring, and $Z(Q)$ is a subring of $Q(D)$.

\section{Automorphisms of $p$-valued groups}

\subsection{$p$-valued groups and $p$-saturated groups}
\label{pVal}
Recall \cite[Definition III.2.1.2]{Laz1965} that a \emph{$p$-valuation} on a group $G$ is a function 
\[\omega : G \to \mathbb{R}_\infty\]
such that for all $x,y \in G$ we have 
\begin{itemize}
\item $\omega(xy^{-1}) \geq \min \{ \omega(x), \omega(y)\}$,
\item $\omega(x^{-1}y^{-1}xy) \geq \omega(x) + \omega(y)$,
\item $\omega(x) = \infty$ if and only if $x = 1$,
\item $\omega(x) > \frac{1}{p-1}$, and
\item $\omega(x^p) = \omega(x) + 1$.
\end{itemize}
The group $G$ is said to be \emph{$p$-valued} if it has a $p$-valuation $\omega$. A \emph{morphism} of $p$-valued groups is a group homomorphism $f : G \to H$ such that $\omega(f(x)) \geq \omega(x)$ for all $x \in G$. 

Define, for each $\nu \in \mathbb{R}$, $G_\nu = \{ g \in G : \omega(g) \geq \nu\}$; this is a normal subgroup of $G$. The group $G$ carries a natural topology which has the $G_\nu$ as a fundamental system of open neighbourhoods of the identity; in this way $G$ becomes a topological group and we say that $G$ is a \emph{complete $p$-valued group} if it is complete with respect to this topology.

Recall \cite[Definition III.2.1.6]{Laz1965} that a complete $p$-valued group $G$ is said to be \emph{$p$-saturated} if the following condition holds:

$\bullet$ if $\omega(x) > 1/(p-1) + 1$, there exists $y \in G$ such that $x = y^p$.

\subsection{Ordered bases}\label{OrdBas} Let $G$ be a complete $p$-valued group. Recall \cite[III.2.2.4]{Laz1965} that an \emph{ordered basis} for $G$ is a subset $\{x_i : i \in I\}$ of $G$ for some totally ordered index set $I$, such that
\begin{itemize}
\item every element $y \in G$ can be written uniquely as a convergent product $y = \prod\limits_{i\in I} x_i^{\lambda_i}$ for some $\lambda_i \in \Zp$, and
\item $\omega(y) = \inf\limits_{i \in I} (\omega(x_i) + v_p(\lambda_i))$.
\end{itemize}
Recall that the \emph{associated graded group} of $G$ is the group
\[\gr G = \bigoplus_{\nu \in \mathbb{R}} G_\nu/G_{\nu^+}\]
where $G_{\nu^+} := \{ g \in G : \omega(g) > \nu\}$. This has the structure of an $\mathbb{R}$-graded $\Fp[\pi]$-Lie algebra, where the action of $\pi$ on homogeneous components is given by $\pi\cdot g G_{\nu^+} = g^pG_{(\nu+1)^+}$. The $p$-valuation $\omega$ on $G$ is said to be \emph{discrete} if $\omega(G \backslash \{1\})$ is a discrete subset of $\mathbb{R}$. 

We say that the complete $p$-valued group $G$ has \emph{finite rank} if it has a finite ordered basis. This property turns out to be independent of the particular $p$-valuation on $G$. 

\begin{lem} Let $H$ be a closed subgroup of a complete $p$-valued group $G$ of finite rank. Then there exists a sequence of integers $n_1 \leq n_2 \leq \cdots \leq n_e$ and an ordered basis $\{g_1,\ldots,g_d\}$ for $G$ such that $\{g_1^{p^{n_1}},g_2^{p^{n_2}}, \cdots, g_d^{p^{n_e}}\}$ is an ordered basis for $H$.
\end{lem}
\begin{proof} 
Consider the $\Fp[\pi]$-modules $\gr H \subseteq \gr G$; by the elementary divisors theorem \cite[Theorem I.1.2.4]{Laz1965} we can find a homogeneous basis $\{\xi_1,\ldots, \xi_d\}$ for $\gr G$ over $\Fp[\pi]$ such that $\{\pi^{n_1}\xi_1, \ldots, \pi^{n_e}\xi_e\}$ is a basis for $\gr H$ over $\Fp[\pi]$ for some $e \leq d$ and some increasing sequence of non-negative integers $n_i$. 

Let $g_i \in G$ be any lift of $\xi_i \in \gr G$. Because the $p$-valuation on $G$ is discrete by \cite[Proposition III.2.2.6]{Laz1965}, we deduce that $\{g_1,\ldots, g_d\}$ is an ordered basis of $G$ and that $\{g_1^{p^{n_1}}, \ldots, g_e^{p^{n_e}}\}$ is an ordered basis for $H$ by applying \cite[Proposition III.2.2.5]{Laz1965}. \end{proof}
Clearly $d = \dim G$ and $e = \dim H$ are the ranks of $H$ and $G$ respectively, and $e = d = \dim G$ if and only if $H$ is open in $G$.

\subsection{Lazard's equivalence of categories}\label{LazEq}
Let $\mathfrak{g}$ be a Lie algebra over $\Zp$. Recall \cite[I.2.2.4]{Laz1965} that $\mathfrak{g}$ is said to be \emph{valued} if there exists a function $w : \mathfrak{g} \to \mathbb{R}_\infty$ satisfying
\begin{itemize}
\item $w(x - y) \geq \min \{w(x), w(y)\}$,
\item $w([x,y]) \geq w(x) + w(y)$,
\item $w(x) = \infty$ if and only if $x = 0$,
\item $w(\lambda x) = v_p(\lambda) + w(x)$
\end{itemize}
for all $x,y \in \mathfrak{g}$ and $\lambda \in \Zp$. The Lie algebra $\mathfrak{g}$ is said to be \emph{saturated} if it is complete with respect to the topology defined by the submodules $\mathfrak{g}_\nu = \{x \in \mathfrak{g} : w(x) \geq \nu\}$ of $\mathfrak{g}$, and the following extra conditions hold:
\begin{itemize}
\item $w(x) > \frac{1}{p-1}$ for all $x \in \mathfrak{g}$, and
\item if $w(x) > 1/(p-1) + 1$, there exists $y \in \mathfrak{g}$ such that $x = py$.
\end{itemize}
A \emph{morphism} of saturated $\Zp$-Lie algebras is a Lie homomorphism $f : \mathfrak{g} \to \mathfrak{g}'$ such that $w(f(x)) \geq w(x)$ for all $x \in \mathfrak{g}$.

Lazard proved \cite[IV.3.2.6]{Laz1965} that there is an isomorphism between the category of $p$-saturated groups and the category of saturated $\Zp$-Lie algebras. Let us recall how this isomorphism works. If $G$ is a $p$-saturated group, let the corresponding saturated $\Zp$-Lie algebra be called $\log(G)$. We view it as a set of formal symbols $\{ \log(g) : g \in G \}$; the $\Zp$-Lie algebra structure on this set is given by the formulas
\[ \begin{array}{cllc} 
\lambda \cdot \log(g) &=& \log(g^\lambda),  & g \in G, \lambda \in \Zp\\
\log(g) + \log(h) &=& \log\left(\lim\limits_{r \to \infty} (g^{p^r}h^{p^r})^{p^{-r}}\right), & g,h \in G\\

[\log(g),\log(h)] &=& \log\left(\lim\limits_{r \to \infty} (g^{p^r}h^{p^r} g^{-p^r}h^{-p^r})^{p^{-2r}}\right) & g,h \in G
\end{array}\]
and the valuation $w$ is given by $w(\log(g)) = \omega(g)$. Conversely, if $\mathfrak{g}$ is a saturated $\Zp$-Lie algebra, let the corresponding $p$-saturated group be called $\exp(\mathfrak{g})$. We view it as a set of formal symbols $\{ \exp(u) = e^u : u \in \mathfrak{g}\}$; the group structure on this set is given by 
\[\begin{array}{cllc}e^u \cdot e^v &=& \exp( \Phi(u,v) )&\quad u,v \in \mathfrak{g},\\
(e^u)^{-1} &=& e^{-u} & \quad u \in \mathfrak{g}
\end{array}\]
and the $p$-valuation is given by $\omega(e^u) = w(u)$ for all $u \in \mathfrak{g}$. Here $\Phi(u,v) = u + v + \frac{1}{2}[u,v] + \cdots $ is the \emph{Baker-Campbell-Hausdorff series}, an infinite series with rational coefficients consisting only of Lie words in $u$ and $v$; see \cite[Th\'eor\`eme IV.3.2.2]{Laz1965}.

\subsection{Transport of structure}\label{TransStr}
Let $f : G \to H$ be an increasing map between two $p$-saturated groups $G$ and $H$ in the sense that
\[\omega(f(g)) \geq \omega(g)\quad\mbox{for all}\quad g \in G\]
but $f$ is not necessarily a group homomorphism. Because $\log : G \to \log(G)$ and $\exp : \log(G) \to G$ are isometries by definition, $f$ induces an increasing map
\[f_\ast := \log\circ f \circ \exp : \log(G) \to \log(H)\]
between the associated saturated $\Zp$-Lie algebras. Similarly if $g : \mathfrak{g} \to \mathfrak{h}$ is an increasing map between two saturated $\Zp$-Lie algebras $\mathfrak{g}$ and $\mathfrak{h}$, then
\[g^\ast = \exp \circ g \circ \log : \exp(\mathfrak{g}) \to \exp(\mathfrak{h})\]
is an increasing map between the associated $p$-saturated groups $\exp(\mathfrak{g})$ and $\exp(\mathfrak{h})$. These notations extend Lazard's equivalence of categories in the sense that $f_\ast$ is the morphism of saturated $\Zp$-Lie algebras associated with a morphism of $p$-saturated groups $f$, and $g^\ast$ is the morphism of $p$-saturated groups associated with a morphism of saturated $\Zp$-Lie algebras $g$.

\subsection{Automorphisms}\label{AutLaz1965} Define the \emph{degree} of an automorphism $\varphi : G \to G$ of a $p$-valued group $G$ by the formula
\[\deg_\omega(\varphi) :=\inf\limits_{g \in G}(\omega(\varphi(g)g^{-1}) - \omega(g)).\]
Thus $\omega( \varphi(g) g^{-1} ) \geq \omega(g) + \deg_\omega(\varphi)$ for all $g \in G$. Note that 
\[\deg_\omega(\varphi) \geq 0 \quad\mbox{ if and only if }\quad \omega(\varphi(g)) \geq \omega(g)\quad \mbox{ for all }\quad g \in G\]
because $\omega(\varphi(g)) \geq \min \{\omega(\varphi(g)g^{-1}) , \omega(g)\}$. Thus $\varphi$ is an increasing map whenever $\deg_\omega(\varphi) \geq 0$, and moreover $\deg_\omega(\varphi) > 0$ forces $\varphi$ to be an isometry. Define
\[\Aut^\omega(G) := \left\{\varphi \in \Aut(G) : \deg_\omega(\varphi) >  \frac{1}{p-1}\right\};\] 
it is not difficult to see that this is a subgroup of $\Aut(G)$. 

Similarly, we define the \emph{degree} of a $\Zp$-linear endomorphism $\sigma : \mathfrak{g} \to \mathfrak{g}$ of a valued $\Zp$-Lie algebra $\mathfrak{g}$ by the formula
\[\deg_w(\sigma) := \inf\limits_{u \in \mathfrak{g}}( w(\sigma(u)) - w(u) ).\]
In this way $A = \End_{\Zp}(\mathfrak{g})$ becomes a valued associative $\Zp$-algebra in the sense of \cite[I.2.2.4]{Laz1965}. Then
\[ \GL^w(\mathfrak{g}) := \left\{\sigma \in A : \deg_w(\sigma - 1) > \frac{1}{p-1}\right\}\]
is a subgroup of the group of units $\GL(\mathfrak{g})$ of $A$ and the map 
\[\sigma \mapsto \deg_w(\sigma - 1)\]
is a \emph{$p$-valuation} on $\GL^w(\mathfrak{g})$ by \cite[Exercise III.3.2.6]{Laz1965}. 

\subsection{Proposition}\label{IsomIsom} Let $G$ be a $p$-saturated group and let $\mathfrak{g} = \log(G)$. The transport of structure map $\varphi \mapsto \varphi_\ast$ defines an isometric monomorphism $\Aut^\omega(G) \hookrightarrow \GL^w(\mathfrak{g})$. 
\begin{proof} The fact that $\varphi_\ast$ is an automorphism of $\mathfrak{g}$ and that $\varphi \mapsto \varphi_\ast$ is a group homomorphism follows from the isomorphism of categories theorem \cite[IV.3.2.6]{Laz1965}. Now
\[ \begin{array}{lll} w(\varphi_\ast(\log(g)) - \log(g)) &=& w(\log (\varphi(g)) + \log (g^{-1}))  \\
&=& \omega\left(\lim\limits_{r \to \infty} (\varphi(g)^{p^r}g^{-p^r})^{p^{-r}}\right).\end{array}\] 
However \cite[Proposition III.2.1.4]{Laz1965} shows that
\[\omega\left( (\varphi(g)^{p^r} g^{-p^r})^{p^{-r}}\right) = \omega( \varphi(g) g^{-1} )\]
for all $r$, so we see that
\[ w(\varphi_\ast(\log(g)) - \log(g)) = \omega( \varphi(g) g^{-1} )\quad\mbox{for all} \quad g \in G.\]
Therefore
\[\begin{array}{lllll} \deg_w(\varphi_\ast - 1) &=& \inf\limits_{u\in \mathfrak{g}}(w(\varphi_\ast(u) - u) - w(u)) &=& \\
& = & \inf\limits_{g \in G}(\omega(\varphi(g)g^{-1}) - \omega(g)) &=& \deg_\omega(\varphi)\end{array}\]
which shows that $\varphi_\ast \in \GL^\omega(\mathfrak{g})$ whenever $\varphi \in \Aut^\omega(G)$. 
\end{proof}

\begin{cor} Let $G$ be a $p$-saturated group. Then $\deg_\omega$ is a $p$-valuation on $\Aut^\omega(G)$ and $\Aut^\omega(G)$ is saturated with respect to this filtration.\end{cor}
\begin{proof}Apply the Proposition and \cite[Exercise III.3.2.6]{Laz1965}. \end{proof}

\subsection{The functor $\Sat$}\label{Sat}
The restriction of a $p$-valuation on a group $G$ to any subgroup of $G$ is again a $p$-valuation, so every subgroup of a $p$-valued group is $p$-valued. In particular, every subgroup of a $p$-saturated group is $p$-valued. Conversely, Lazard shows in \cite[III.3.3.1]{Laz1965} that if $G$ is a $p$-valued group then there exists an isometric inclusion $\iota_G : G \to \Sat(G)$ into a $p$-saturated group $\Sat(G)$, and that $\Sat(G) = G$ if and only if $G$ is $p$-saturated. Moreover every morphism $f : G \to H$ of $p$-valued groups extends to a unique morphism $\Sat(f) : \Sat(G) \to \Sat(H)$ making $\Sat$ into a functor. Thus $p$-valued groups are precisely the subgroups of $p$-saturated groups.

\begin{lem} Let $G$ be a complete $p$-valued group of finite rank. Then $f \mapsto \Sat(f)$ is an isometric embedding of $\Aut^\omega(G)$ into $\Aut^\omega(\Sat(G))$. \end{lem}
\begin{proof}
Let $\tilde{G} = \Sat(G)$ and let $\tilde{\varphi} = \Sat(\varphi)$ be the extension of $\varphi \in \Aut^\omega(G)$ to $\Aut(\tilde{G})$. Since $\iota_G:G \to \Sat(G)$ is an isometry, we will view $G$ as a subgroup of $\tilde{G}$ and denote the $p$-valuation on $\tilde{G}$ by the same letter $\omega$. 

Now clearly $\deg_\omega(\tilde{\varphi}) \leq \deg_\omega(\varphi)$. To see that the reverse inequality holds, let $g \in \tilde{G}$. Because $G$ has finite rank, \cite[Theorem IV.3.4.1]{Laz1965} tells us that we can find $n \in \mathbb{N}$ such that $g^{p^n} \in G$. Now by \cite[Proposition III.2.1.4]{Laz1965},
\[\begin{array}{lll} \omega(\tilde{\varphi}(g)g^{-1}) &=& \omega(\tilde{\varphi}(g)^{p^n}g^{-p^n}) - n \\
&=& \omega( \varphi( g^{p^n} ) g^{-p^n} ) - n\geq \\
&\geq & \deg_\omega(\varphi) + \omega( g^{p^n} ) -n = \deg_\omega(\varphi) + \omega(g).\end{array}\]
because $\tilde{\varphi}_{|G} = \varphi$. Therefore
\[\omega(\tilde{\varphi}(g)g^{-1}) \geq \deg_\omega(\varphi) + \omega(g)\quad \mbox{for all}\quad g \in \tilde{G},\]
and $\deg_\omega(\tilde{\varphi}) \geq \deg_\omega(\varphi)$. 
\end{proof}

\begin{cor} Let $G$ be a complete $p$-valued group of finite rank. Then
\begin{enumerate}[{(}a{)}]
\item $\deg_\omega$ is a $p$-valuation on $\Aut^\omega(G)$, and
\item $\Aut^\omega(G)$ is torsion-free.
\end{enumerate}\end{cor}
\begin{proof}(a) This follows from Corollary \ref{IsomIsom}.

(b) This is clear.
\end{proof}
\subsection{The logarithm of an automorphism}\label{LogAut}
Let $G$ be a $p$-saturated group. By Proposition \ref{IsomIsom}, $\deg_w(\varphi_\ast - 1) > 1/(p-1)$ for any $\varphi \in \Aut^\omega(G)$. Hence the logarithm series
\[\sum_{k=1}^\infty \frac{(-1)^{k+1}}{k} (\varphi_\ast - 1)^k\]
converges to an element $\log \varphi_\ast$, say, inside $\End_{\Zp}(\mathfrak{g})$ by \cite[III.1.1.5]{Laz1965}. It is easy to see that
\[ w( (\log \varphi_\ast)(u) ) \geq w(u) + \deg_\omega(\varphi)\]
so $\log \varphi_\ast $ is an increasing function $\mathfrak{g}\to \mathfrak{g}$ and we can transport it back to $G$.

\begin{defn} Let $G$ be a $p$-saturated group and let $\varphi \in \Aut^\omega(G)$. Define the \emph{logarithm} of $\varphi$ by the formula
\[z(\varphi) = (\log \varphi_\ast)^\ast : G \to G.\]
\end{defn}

Recalling the notation of $\S \ref{TransStr}$, we see that $z(\varphi)$ is \emph{a priori} just an increasing map $z(\varphi) : G \to G$. We will shortly see that in some cases, there is a way of defining $z(\varphi)$ more directly using the group structure on $G$.

\subsection{Automorphisms trivial mod centre}\label{AutModZ}
Let $G$ be a $p$-saturated group and let $Z$ be its centre. Let $\Aut^\omega_Z(G)$ denote the subgroup of $\Aut^\omega(G)$ which consists of automorphisms that induce the trivial automorphism of $G/Z$. Equivalently,
\[\Aut^\omega_Z(G) = \{ \varphi \in \Aut^\omega(G) : \varphi(g)g^{-1} \in Z \quad\mbox{for all}\quad g \in G.\}\]

\begin{prop} Let $G$ be a $p$-saturated group and let $\varphi\in\Aut^\omega_Z(G)$. 
\begin{enumerate}[{(}a{)}]
\item For all $g\in G$ and all $r \geq 0$, there exists $\epsilon_r(g) \in G$ such that
\[\varphi^{p^r}(g)g^{-1} = z(\varphi)(g)^{p^r}\epsilon_r(g)^{p^{2r}}.\]
\item $z(\varphi)(g) = \lim\limits_{r \to \infty} (\varphi^{p^r}(g) g^{-1})^{p^{-r}}$ for all $g \in G$.
\item $z(\varphi)$ is a group homomorphism from $G$ to $Z$.
\end{enumerate}
\end{prop}
\begin{proof} (a) Using transport of structure, let us compare the expressions
\[ \varphi^{p^r}(g) g^{-1}\quad\mbox{and}\quad z(\varphi)(g)^{p^r} .\]
Write $g = e^u \in G$ for some $u \in \mathfrak{g} = \log(G)$, and $\sigma = \varphi_\ast = e^\alpha$ for some $\alpha \in \End_{\Zp}(\mathfrak{g})$. Then
\[ \log(z(\varphi)(e^u)^{p^r}) = p^r(\log \sigma)(u) = p^r \alpha(u),\]
whereas
\[ \log(\varphi^{p^r}(e^u) e^{-u}) = \log(e^{\sigma^{p^r}(u)} e^{-u}) = \Phi( e^{p^r \alpha}(u), -u).\]
Now because $\varphi^{p^r} \in \Aut^\omega_Z(G)$, $\sigma^{p^r} - 1$ maps $\mathfrak{g}$ into $Z(\mathfrak{g})$ and therefore 
\[ [e^{p^r \alpha}(u), -u] = [e^{p^r \alpha}(u) - u, -u] = 0.\]
Therefore
\[ \Phi( e^{p^r \alpha}(u), -u) = e^{p^r \alpha}(u) - u = p^r\alpha(u) + p^{2r} \beta_r(u)\]
for some $\beta_r(u) \in \mathfrak{g}$. So
\[ \log(z(\varphi)(g)^{-p^r} \varphi^{p^r}(g) g^{-1}) = \Phi(-p^r \alpha(u), p^r \alpha(u) + p^{2r} \beta_r(u)) \in p^{2r} \mathfrak{g}\]
and part (a) follows.

(b) The above computation shows that 
\[\log( (\varphi^{p^r}(g) g^{-1})^{p^{-r}} ) = \alpha(\log(g)) + p^r\beta_r(\log(g)) = \log( z(\varphi)(g) ) + p^r \beta_r(\log(g))\]
for all $g \in G$. Therefore
\[ \lim\limits_{r \to \infty}\log( (\varphi^{p^r}(g) g^{-1})^{p^{-r}} ) = \log ( z(\varphi)(g) ) \]
for all $g \in G$. Part (b) now follows because $\log : G \to \mathfrak{g}$ is a homeomorphism.

(c) For each $r \geq 0$ and $g, h \in G$ we have
\[ \varphi^{p^r}(gh) (gh)^{-1} = \varphi^{p^r}(g) \varphi^{p^r}(h) h^{-1} g^{-1} = \varphi^{p^r}(g) g^{-1} \cdot \varphi^{p^r}(h) h^{-1}\]
because $\varphi^{p^r}(h) h^{-1}$ is central in $G$ by assumption on $\varphi$. So $g \mapsto \varphi^{p^r}(g) g^{-1}$ is a group homomorphism $G \to Z$ for all $r$. Now take limits.
\end{proof}

\subsection{Proposition}\label{SatIsNice} Let $G$ be a $p$-valued group of finite rank, let $\varphi \in \Aut^\omega_Z(G)$ and let $\tilde{\varphi}$ be the extension of $\varphi$ to $\tilde{G} = \Sat(G)$. Then $\tilde{\varphi}$ is also trivial mod centre.
\begin{proof} By \cite[Theorem IV.3.4.1]{Laz1965}, we can find an integer $n$ such that $\tilde{G}^{p^n} \leq G$. Because $\tilde{G}$ is torsion-free, it follows that $Z(G) \leq Z(\tilde{G})$. Hence $\varphi(g)g^{-1} \in Z(\tilde{G})$ for all $g \in \tilde{G}^{p^n}$. 

Now fix $g \in \tilde{G}$ and consider $\log(\tilde{\varphi}(g)) - \log(g) \in \log(\tilde{G})$. We have
\[p^n (\log(\tilde{\varphi}(g)) - \log(g)) = \log(\varphi(g^{p^n})) - \log(g^{p^n}) = \log( \lim\limits_{r \to \infty} (\varphi(g^{p^{n+r}}) g^{-p^{n+r}})^{p^{-r}}) \]
which lies in $\log(Z(\tilde{G}))$ because $\tilde{G}/Z(\tilde{G})$ is torsion-free. Since $\log(\tilde{G})$ is a torsion-free $\Zp$-module we deduce that $\log(\tilde{\varphi}(g)) - \log(g) = \log(z) $ for some $z \in Z(\tilde{G})$, and therefore $\tilde{\varphi}(g) = gz$ because $\log$ is a bijection. Thus $\tilde{\varphi}$ is trivial mod centre as claimed.
\end{proof}

\section{$\Gamma$-primes and open subgroups}
\label{Gamma}

\subsection{Completed group rings}
\label{ComplGpRng}
From now on, $k$ will denote an arbitrary field of characteristic $p$ and $G$ will denote a compact $p$-adic analytic group. Let $kG$ denote the \emph{completed group ring} of $G$ with coefficients in $k$:
\[ kG := \invlim k[G/U]\]
where the inverse limit is taken over all the open normal subgroups $U$ of $G$. 

\begin{lem} Let $H$ be a closed subgroup of $G$, and let $I_1, \ldots, I_m, J$ be right ideals of $kH$. Then
\be
\item $I_1kG \cap \cdots \cap I_m kG = (I_1\cap \cdots \cap I_m)kG$, and
\item $JkG \cap kH = J$.
\ee\end{lem}
\begin{proof} The proof of \cite[Lemma 5.1]{Ard2006} shows that $kG$ is a faithfully flat $kH$-module. Now part (a) follows by applying the $- \otimes_{kH}kG$ functor to the exact sequence $0 \to I_1 \cap \cdots \cap I_m \to kH \to \oplus_{j=1}^m kH/I_j$, and part (b) follows by applying \cite[Lemma 7.2.5]{MCR} to the $kH$-module $kH/J$.
\end{proof}

\subsection{$I^\dag$ and $I^\chi$}\label{ContSub}
Let $I$ be a right ideal in $kG$ and recall the controller subroup $I^\chi$ of $I$ from $\S \ref{ContSubIntro}$. Following Roseblade, we define another subgroup associated to $I$ as follows:
\[ I^\dag := (1 + I)\cap G\]
and say that $I$ is \emph{faithful} precisely when $I^\dag$ is trivial. Since $kG$ has a Zariskian filtration that generates its natural topology, every right ideal $I$ is closed. Since the natural map $G \to kG$ that sends $g \to g-1$ is continuous and $I^\dag$ is the preimage of $I$ under this map, we see that $I^\dag$ is always a closed subgroup of $G$.

\begin{lem} Let $I$ be a right ideal of $kG$ and let $\varphi \in \Aut(G)$. Suppose that the extension of $\varphi$ to an algebra automorphism of $kG$ preserves $I$. Then 
\be 
\item $\varphi$ preserves both $I^\dag$ and $I^\chi$, and
\item $I^\dag$ is contained in $I^\chi$ whenever $I \neq kG$.
\ee
\end{lem}
\begin{proof} (a) Since $kG$ is Noetherian, the ascending chain $I \subseteq \varphi^{-1}(I) \subseteq \varphi^{-2}(I) \subseteq \cdots $ terminates so $\varphi^{-1}$ preserves $I$. Applying $\varphi^{-1}$ to the equation $I = (I \cap kI^\chi)kG$ shows that $\varphi^{-1}(I^\chi)$ controls $I$ and therefore contains $I^\chi$. Hence $\varphi(I^\chi) \subseteq I^\chi$, and $\varphi(I^\dag) \subseteq I^\dag$ is clear.

(b) Choose an open subgroup $U$ that controls $I$, let $\{g_1 = 1,\ldots,g_m\}$ be a complete set of right coset representatives for $U$ in $G$ and let $x \in I^\dag$. Then $x - 1 \in I$ so $x - 1 = \sum_{i=1}^mr_ig_i$ for some $r_i \in I \cap kU$. Since $I$ is a proper right ideal by assumption, equating coefficients shows that $x$ must lie in $U$, since otherwise $-1 = r_1 \in I$. Hence $I^\dag \subseteq U$ for every open subgroup $U$ that controls $I$ and the result follows.
\end{proof}

\subsection{Isolated subgroups}
\label{IsoSub}
We say that a closed normal subgroup $H$ of a complete $p$-valued group $G$ is \emph{isolated} if $G/H$ is torsion-free.

\begin{lem} Let $I$ be a two-sided ideal of $kG$. 
\be\item $I^\dag$ and $I^\chi$ are closed \emph{normal} subgroups of $G$.
\item If $I$ is semiprime and $G$ is pro-$p$, then $G/I^\dag$ has no non-trivial finite normal subgroups.
\item If $I$ is semiprime and $G$ is $p$-valued, then $I^\dag$ is isolated.
\ee \end{lem}
\begin{proof}
(a) Lemma \ref{ContSub}(a) shows that $I^\chi$ is stable under every inner automorphism of $G$, so $I^\chi$ is normal. It is clear that $I^\dag$ is also normal.

(b) Let $N / I^\dag$ be a finite normal subgroup of $G/I^\dag$. Since $G$ is pro-$p$, $N/I^\dag$ is a finite $p$-group, so some power of the augmentation ideal of $k[N/I^\dag]$ is zero. Hence $((N - 1)kG)^a \subseteq I$ for some integer $a$, but $I$ is semiprime so in fact $(N-1)kG \subseteq I$. Therefore $N \leq I^\dag$ and $N/I^\dag$ is trivial.

(c) By \cite[IV.3.4.1]{Laz1965}, $\Sat(I^\dag) \cap G$ is a closed normal subgroup of $G$ containing $I^\dag$ as a subgroup of finite index, so $I^\dag = \Sat(I^\dag) \cap G$ by part (b). Hence $G/I^\dag$ embeds into $\Sat(G) / \Sat(I^\dag)$, which is $p$-saturated by \cite[III.3.3.2.4]{Laz1965} and therefore torsion-free.\end{proof}
The prime ideal $\langle x - y^p \rangle$ of the commutative power series ring $k[[x,y]]$ shows that the controller subgroup need not be isolated, in general.

\subsection{$\Gamma$-prime ideals}\label{GammaPrime}
Let $\Gamma$ be a group acting on $G$ by automorphisms. We say that an ideal $P$ of $kG$ is \emph{$\Gamma$-prime} if $P$ is $\Gamma$-invariant, and whenever $I,J$ are $\Gamma$-invariant ideals of $kG$ such that $IJ \subseteq P$, we have either $I \subseteq P$ or $J \subseteq P$.

\begin{lem} Let $P$ be a $\Gamma$-prime ideal of $kG$. 
\begin{enumerate}[{(}a{)}]
\item $P$ is semiprime.
\item The minimal primes $P_1,\ldots, P_m$ above $P$ form a single $\Gamma$-orbit.
\end{enumerate}
\end{lem}
\begin{proof} (a) Clearly the prime radical $\sqrt{P}$ of $P$ is $\Gamma$-invariant, and $\sqrt{P}^n \subseteq P$ for some $n$ since $kG$ is Noetherian. Therefore $P = \sqrt{P}$ as $P$ is $\Gamma$-prime.

(b) Let $P_1,\ldots, P_\ell$ be the $\Gamma$-orbit of $P_1$. If $\ell < m$ then $I := \cap_{i \leq \ell} P_i$ and $J := \cap_{i > \ell} P_i$ are $\Gamma$-invariant ideals and $I \cap J = P$ since $P = P_1 \cap \cdots\cap P_m$ is semiprime by part (a). Since $P$ is $\Gamma$-prime, either $I \subseteq P$ or $J \subseteq P$. If $I \subseteq P$ then $P_1 \cdots P_\ell \subseteq P \subseteq P_m$ forces $P_m$ to be equal to one of the $P_i$ for some $i \neq m$, a contradiction. $J \subseteq P$ is similarly impossible, so $\ell = m$ and $\Gamma$ acts transtively on the $P_i$.
\end{proof}
If $B$ is a subring of a commutative ring $A$ and $P$ is a prime ideal of $A$ then $P \cap B$ is always a prime ideal of $B$. In the non-commutative setting, $P \cap B$ will in general not be a prime ideal; it may even fail to be semiprime. However for completed group rings (and for group algebras of polycyclic groups) we have the following positive result.

\begin{prop}
Let $P$ be a prime ideal of $kG$ and let $N$ be a closed normal subgroup of $G$. Then $P \cap kN$ is a $G$-prime ideal of $kN$. In particular $P \cap kN$ is semiprime.
\end{prop}
\begin{proof} Let $I, J$ be $G$-invariant ideals of $kN$ with $IJ \subseteq P \cap kN$. Then $IkG$ and $JkG$ are two-sided ideals in $kG$ whose product is contained in $P$, so without loss of generality we may assume that $IkG \subseteq P$. Therefore $I = IkG \cap kN \subseteq P \cap kN$ by Lemma \ref{ComplGpRng}(b) and the result follows.
\end{proof}

\subsection{Non-splitting primes}\label{NonSplit} To prove our analogue of Zalesskii's Theorem for a prime ideal $P$, we would like to first reduce to the case when $P^\chi = G$. Since $(P \cap kP^\chi)^\chi = P^\chi$, it is tempting to try to replace $P$ by $P\cap kP^\chi$. However $P \cap kP^\chi$ will not in general be a prime ideal.

\begin{defn} Let $P$ be a prime ideal of $kG$. We say that $P$ is \emph{non-splitting} if $P \cap kU$ is again prime for any open normal subgroup $U$ of $G$ that controls $P$.
\end{defn}
The reason for this definition is the following
\begin{prop} Let $P$ be a non-splitting prime ideal of $kG$. Then $P \cap kP^\chi$ is a prime ideal of $kP^\chi$.
\end{prop}
\begin{proof} Since $P$ is a two-sided ideal, $P^\chi$ is a closed normal subgroup of $G$. Let $P_1, \ldots, P_m$ be the minimal primes over $P \cap kP^\chi$. Since $P \cap kP^\chi$ is $G$-prime by Proposition \ref{GammaPrime}, the conjugation action of $G$ on the $P_i$ is transitive by Lemma \ref{GammaPrime}(b). Let $U$ be the kernel of this action; then $U$ is an closed normal subgroup of $G$ of finite index and therefore also open. Moreover $U$ contains $P^\chi$ since the $P_i$ are two-sided ideals in $kP^\chi$, so $P \cap kU$ is prime by assumption. Now
\[P \cap kU = (P \cap kP^\chi)kU = P_1 kU \cap \cdots \cap P_mkU\]
by Lemma \ref{ComplGpRng}(a), and the $P_ikU$ are two-sided ideals in $kU$ by the definition of $U$. Since $P \cap kU$ is prime, $P_ikU =  P \cap kU$ for some $i$ and therefore
\[P \cap kP^\chi = P \cap kU \cap kP^\chi = (P_i kU) \cap kP^\chi = P_i\]
by Lemma \ref{ComplGpRng}(b). Hence $P \cap kP^\chi = P_i$ is prime.\end{proof}
\subsection{Essential decompositions}\label{EssDec}
\begin{defn} Let $A$ be a ring and let $J_1,\ldots, J_m$ be proper right ideals of $A$ with intersection $I$. 
\be \item We say that $I = J_1 \cap \cdots \cap J_m$ is an \emph{essential decomposition} of $I$ if the natural embedding
$\frac{A}{I} \hookrightarrow \frac{A}{J_1} \oplus \frac{A}{J_2} \oplus \cdots \oplus \frac{A}{J_m}$ has essential image. 
\item If $H$ is a subgroup of the group of units of $A$ then we call the decomposition \emph{$H$-invariant} if $H$ acts transtively by conjugation on the right ideals $J_i$.
\ee \end{defn}

It follows from the definition of \emph{uniform dimension} \cite[\S 2.2.10]{MCR} that
\[\udim (A/I) = \sum_{i=1}^m \udim(A / J_i)\]
whenever $I = J_1 \cap \cdots \cap J_m$ is an essential decomposition of $I$. This implies that the number of terms $m$ in any essential decomposition of $I$ is bounded above by $\udim(A/I)$.

\begin{example} Let $A$ be a semiprime Noetherian ring and let $P_1, \ldots, P_m$ be the minimal primes of $A$. Then $0 = P_1 \cap \cdots \cap P_m$ is an essential decomposition of the zero ideal.
\end{example}
\begin{proof} Let $A' = (A/P_1) \oplus \cdots \oplus (A/P_m)$ and let $Q$ be the classical ring of quotients of $A$. Then there exists a unit $q \in Q$ such that $qA' \subseteq A \subseteq A'$ by \cite[Proposition 3.2.4(iii)]{MCR}. By clearing denominators we may assume that $q \in A$ is a regular element. Suppose that $M$ is an $A$-submodule of $A'$ such that $A \cap M = 0$. Then $qA \cap qM = 0$, but $qA \cong A$ as a right ideal so $\udim(qA) = \udim(A)$ and therefore $qA$ is essential in $A$ by \cite[Corollary 2.2.10(iii)]{MCR}. Hence $qM = 0$, but $q$ is regular so $M = 0$ and $A$ is essential in $A'$.
\end{proof}

\subsection{Virtually prime right ideals}
\label{VirtNonSplit}
\begin{defn} Let $I$ be a right ideal of $kG$. We say that $I$ is \emph{virtually prime} if $I = PkG$ for some prime ideal $P$ of $kU$ for some open subgroup $U$ of $G$. If in addition $P$ is non-splitting, then we say that $I$ is \emph{virtually non-splitting}.
\end{defn}

Clearly every prime ideal is virtually prime as a right ideal, and every non-splitting prime ideal is a virtually non-splitting right ideal.

\begin{lem} Suppose that $G$ is a pro-$p$ group, let $V$ be an open subgroup of $G$ and let $M$ be a $kV$-module. If $N$ is an essential $kV$-submodule of $M$ then $N \otimes_{kV} kG$ is an essential $kG$-submodule of $M \otimes_{kV} kG$.
\end{lem}
\begin{proof} By an easy induction on the index of $V$ in $G$, we are reduced to the case when $V$ is maximal in $G$. Because $G$ is pro-$p$, $V$ is normal in $G$. Now $M \otimes_{kV} kG$ is isomorphic as a $kV$-module to a finite direct sum of twists $Mg$ of $M$, as $g$ ranges over a set of coset representatives for $V$ in $G$. Since $Ng$ is essential in $Mg$ for all $g \in G$, $N \otimes_{kV} kG$ is essential in $M \otimes_{kV}kG$ by \cite[Lemma 2.2.2(iv)]{MCR} as a $kV$-module, and therefore per force also as a $kG$-module.
\end{proof}

We now present a method of constructing virtually non-splitting right ideals starting from arbitrary prime ideals.

\begin{thm} Suppose that $G$ is a pro-$p$ group, let $P$ be a prime ideal of $kG$ and let $P = I_1 \cap I_2 \cap \cdots \cap I_m$ be a $G$-invariant essential virtually prime decomposition of $P$ with $m$ as large as possible. Then each $I_j$ is virtually non-splitting.
\end{thm}
\begin{proof} By symmetry, it is enough to prove that $I := I_1$ is virtually non-splitting. Choose an open subgroup $U$ of $G$ such that $P := I \cap kU$ is prime and such that $I = PkG$. Let $V$ be an open normal subgroup of $U$ which controls $P$ and let $Q_1,\ldots, Q_r$ be the minimal primes above $P \cap kV$. Then $(Q_i kG) \cap kV = Q_i$ for each $i$ by Lemma \ref{ComplGpRng}(b), so each $Q_i kG$ is virtually prime and we obtain a virtually prime decomposition
\[I = (P \cap kV)kG = (Q_1 \cap \cdots \cap Q_r)kG = Q_1kG \cap \cdots \cap Q_r kG\]
by applying Lemma \ref{ComplGpRng}(a). Since $P \cap kV$ is semiprime by Proposition \ref{GammaPrime}, $k V / P \cap kV$ is an essential $kV$-submodule of $(kV / Q_1) \oplus \cdots \oplus (kV / Q_r)$ by Example \ref{EssDec}. Therefore $kG / I$ is an essential $kG$-submodule of $(kG/Q_1kG)\oplus \cdots \oplus (kG / Q_rkG)$ by the Lemma. Since our original decomposition of $P$ was $G$-invariant, we can find $g_j \in G$ such that $I_j = {}^{g_j}I$ for each $j$, and then the composite embedding
\[\frac{kG}{P} \hookrightarrow \bigoplus_{j=1}^m \frac{kG}{I_j} \hookrightarrow \bigoplus_{j=1}^m \bigoplus_{i=1}^r \frac{kG}{\left({}^{g_j}Q_i\right)kG} \]
has essential image. Therefore
\[P = \bigcap_{j=1}^m \bigcap_{i=1}^r \left({}^{g_j}Q_i\right)kG\]
is another essential virtually prime decomposition of $P$. Because $U$ acts transitively on the $Q_i$ by Lemma \ref{GammaPrime}(b), we see that $G$ acts transitively on the ${}^{g_j}Q_ikG$, so this decomposition is also $G$-invariant. The maximality of $m$ now forces $r = 1$, so $P\cap kV$ is prime for any open normal subgroup $V$ of $U$. Therefore $P = I \cap kU$ is a non-splitting prime and $I$ is virtually non-splitting.
\end{proof}

\subsection{Orbital subgroups}\label{GammaOrb}
Let the group $\Gamma$ act on a set $X$. Imitating Roseblade \cite[\S 1.3]{Roseblade}, we say that an element $x \in X$ is \emph{$\Gamma$-orbital} if the $\Gamma$-orbit of $x$ is finite, and that a profinite group $G$ is \emph{orbitally sound} if for any closed $G$-orbital subgroup $H$ of $G$, the intersection $H^\circ$ of all $G$-conjugates of $H$ has finite index in $H$. 

\begin{thm} Let $G$ be a torsion-free, orbitally sound, pro-$p$, $p$-adic analytic group. Let $A$ be a closed subgroup of $G$ such that every faithful virtually non-splitting right ideal $I$ of $kG$ is controlled by $A$. Then every faithful prime ideal $P$ of $kG$ is also controlled by $A$.
\end{thm}\begin{proof} Since $kG/P$ is Noetherian, its uniform dimension provides an upper bound to the number of terms in any essential decomposition of $P$. Since $P$ is prime, it is virtually prime as a right ideal, so $P = P$ is a $G$-invariant essential virtually prime decomposition of $P$. Choose such a decomposition $P = I_1 \cap \cdots \cap I_m$ with $m$ as large as possible. Fix $j$; then $I_j$ is virtually non-splitting by Theorem \ref{VirtNonSplit} and $(I_j^\dag)^\circ = P^\dag = 1$ because $P$ is faithful, so the $G$-orbital subgroup $I_j^\dag$ is finite since $G$ is orbitally sound. But $G$ is torsion-free so $I_j^\dag = 1$ and each $I_j$ is faithful. Therefore every $I_j$ is controlled by $A$ by assumption and thus $P$ is also controlled by $A$.
\end{proof}

This result is very useful. As we will see in \S\ref{Zal}, it allows us to assume that the ideal $P \cap kP^\chi$ is actually \emph{prime} and not just $G$-prime, after the minor inconvenience of passing to an open subgroup of $G$. By replacing $G$ by $P^\chi$ and $P$ by $P\cap kP^\chi$, we may then focus on showing that a faithful prime $P$ of $kG$ which is not controlled by any proper subgroup of $G$ must be \emph{rigid}: it cannot be stabilized by any sufficiently nice non-trivial automorphism in $\Aut^\omega(G)$. We expect Theorem \ref{GammaOrb} to come in useful in future work on prime ideals in Iwasawa algebras. 

The wide applicability of Theorem \ref{GammaOrb} is guaranteed by our next result.

\subsection{Proposition}
Every complete $p$-valued group $G$ of finite rank is orbitally sound.
\begin{proof} Let $H$ be a closed $G$-orbital subgroup of $G$. Since $G$ has finite index in its saturation by \cite[IV.3.4.1]{Laz1965}, we may assume that $G$ is $p$-saturated. We will now show that $\tilde{H} := \Sat(H)$ is normal in $G$.

Let $g \in G$ and let $\varphi \in \Aut^\omega(G)$ be the conjugation action of $g$.  Because $H$ is $G$-orbital and $G$ is pro-$p$, $\varphi^{p^m}$ stabilizes $H$ for some integer $m$, so $\varphi^{p^m}$ also stabilizes $\tilde{H}$. By Lazard's equivalence of categories $\S \ref{LazEq}$, $(\varphi^{p^m})_\ast = (\varphi_\ast)^{p^m}$ stabilizes the Lie subalgebra $\mathfrak{h} := \log (\tilde{H})$ of $\mathfrak{g} := \log(G)$. Since $\varphi \in \Aut^\omega(G)$, we can consider the logarithm $\psi := \log \varphi_\ast : \mathfrak{g}\to \mathfrak{g}$ of $\varphi_\ast$ defined in $\S \ref{LogAut}$. Now
\[ p^m \psi(\mathfrak{h}) = \log( (\varphi_\ast)^{p^m} ) ( \mathfrak{h} )\subseteq \mathfrak{h} \]
so $\psi$ preserves $\mathfrak{h}$ since $\mathfrak{h}$ is a saturated Lie algebra. Hence $\varphi_\ast = \exp(\psi)$ also preserves $\mathfrak{h}$  and therefore $\varphi = (\varphi_\ast)^\ast$ stabilizes $\tilde{H}$. 

Thus $\tilde{H}$ is normal in $G$ as claimed, and its open subgroups $\tilde{H}^{p^n}$ are also normal in $G$ for all $n$. But $H$ contains one of these subgroups by \cite[IV.3.4.1]{Laz1965}, $\tilde{H}^{p^r}$ say, so 
\[\tilde{H}^{p^r} = \left(\tilde{H}^{p^r}\right)^\circ \leq H^\circ.\]
Hence $H^\circ$ is open in $H$. 
\end{proof}

\section{The Mahler expansion of an automorphism}
\label{MahlerChapter}

\subsection{Rational $p$-valuations}\label{RatVal}
From now on, $G$ will denote a complete $p$-valued group of rank $d$. 

By \cite[Proposition III.3.1.11]{Laz1965}, $G$ has a $p$-valuation $\omega$ which takes takes rational values. In fact, by \cite[Lemma 7.3]{CoaSchSuj2003} it is possible to find a $p$-valuation $\omega$ on $G$ and an integer $e$ such that 
\begin{itemize} 
\item $\omega(g) \in e^{-1} \mathbb{Z}$ for all $1 \neq g \in G$, and
\item $\gr G$ is an abelian $\Fp[\pi]$-Lie algebra.
\end{itemize}
We will henceforth fix such a $p$-valuation $\omega$ on $G$. Until the end of $\S \ref{MahlerChapter}$, we also fix an ordered basis $\mathbf{g} := \{g_1,\ldots, g_d\}$ for $G$ in the sense of $\S \ref{OrdBas}$. Whenever $\alpha \in \mathbb{N}^d$ is a multi-index, define 
\[\wt{\alpha} := \sum_{i=1}^d \alpha_i \omega(g_i).\]
We also define \emph{coordinates of the second kind} on $G$ to be the function
\[ \begin{array}{llll} \theta_{\mathbf{g}} :&  G & \to & \Zp^d \\
& \mathbf{g}^\lambda & \mapsto & \lambda.
\end{array}\]
Let $b_i = g_i - 1 \in k[G]$ for each $i$, and write
\[ \mathbf{b}^\alpha = b_1^{\alpha_1} \cdots b_d^{\alpha_d} \in k[G] \quad\mbox{for each}\quad \alpha \in \mathbb{N}^d.\]

\subsection{The valuation $w$ on $k[G]$}\label{valkG}
Recall the associated graded group $\gr G$ of $G$ from $\S \ref{OrdBas}$, let $\overline{\gr G}$ denote the $\Fp$-vector space $\gr G / \pi \cdot \gr G$ and let $\overline{\xi}$ denote the image of $\xi \in \gr G$ in $\overline{\gr G}$.

\begin{lem} \be
\item There is a filtration $w$ on $k[G]$ such that 
\[ \gr k[G] \cong \Sym (\overline{\gr G} \otimes_{\Fp} k).\]
\item $\gr k[G]$ can be identified with the polynomial algebra $k[X_1,\ldots, X_d]$ where
$X_i = \gr b_i \in \gr kG$ has degree $w(b_i) = \omega(g_i)$ for all $i$. 
\item $w ( \mathbf{b}^\alpha )= \wt{\alpha}$ for all $\alpha \in\mathbb{N}^d$.
\item The completion of $k[G]$ with respect to the filtration $w$ is isomorphic to $kG$. 
\ee\end{lem}
\begin{proof}(a) When $k$ is the finite field $\Fp$, this follows from \cite[Theorem III.2.3.3]{Laz1965}; the general case follows from an easy extension of scalars argument. 

(b) The following are equivalent by \cite[Proposition III.2.2.5]{Laz1965}:
\begin{itemize}
\item $\{g_1,\ldots,g_d \}$ is an ordered basis of $G$,
\item $\{\gr g_1,\ldots,\gr g_d\}$ is a basis for $\gr G$ as an $\Fp[\pi]$-module,
\item $\{\overline{\gr g_1},\ldots,\overline{\gr g_d}\}$ is an $\Fp$-basis for $\overline{\gr G}$.
\end{itemize}
But $\gr b_i \in \gr k[G]$ corresponds to $\overline{\gr g_i}$ in the isomorphism of part (a).

(c) The polynomial algebra $k[X_1,\ldots, X_d]$ has no zero-divisors, so the filtration $w$ is a valuation. Now apply part (b). 

(d) This follows from the proof of \cite[Theorem III.2.3.3]{Laz1965}. \end{proof}

\begin{cor}\be \item Every element of $kG$ can be written as a convergent power series in $b_1,\ldots, b_d$:
\[kG = \left\{ \sum_{\alpha \in \mathbb{N}^d} \lambda_\alpha \mathbf{b}^\alpha : \lambda_\alpha \in k \right\}.\]
\item The extension of the valuation $w$ to $kG$ is given by 
\[w\left(\sum_{\alpha \in \mathbb{N}^d} \lambda_\alpha \mathbf{b}^\alpha\right) = \inf \{ \wt{\alpha} : \lambda_\alpha \neq 0 \}.\]
\ee\end{cor}
\begin{proof} View $k$ as a complete filtered ring with the trivial filtration $v$ given by $v(\lambda) = 0$ if $\lambda \neq 0$ and $v(0) = \infty$. Then $kG$ is a complete filtered $k$-module and $\gr kG$ is free over $\gr k$ by the Lemma. The valuation $w$ on $kG$ is discrete because the $p$-valuation $\omega$ on $G$ takes values in $e^{-1} \mathbb{Z}$ by construction, so the result follows from \cite[Th\'eor\`eme I.2.3.17]{Laz1965}.
\end{proof}
\subsection{Mahler's Theorem}
\label{Mahler}
For each multiindex $\alpha \in \mathbb{N}^d$, there is a continuous function
\[\begin{array}{llll} \binom{-}{\alpha} :& \Zp^d &\to& \Zp \\
& \lambda &\mapsto & \binom{\lambda}{\alpha} := \binom{\lambda_1}{\alpha_1} \cdots \binom{\lambda_d}{\alpha_d}.\end{array}
\]
It turns out that these binomial coefficients form a nice topological basis for the space of continuous functions $C(\Zp^d,\Zp)$. More generally, Mahler's Theorem \cite[III.1.2.4]{Laz1965} states the following:

\begin{thm} Let $M$ be a complete $\Zp$-module and let $f : \Zp^d \to M$ be a continuous function. Then there is a collection of elements $\{C_\alpha(f) \in M : \alpha \in \mathbb{N}^d\}$ depending only on $f$ such that
\begin{itemize}
\item $C_\alpha(f) \to 0$ as $\alpha \to \infty$,
\item $f(\lambda) = \sum_{\alpha\in\mathbb{N}^d} C_\alpha(f) \binom{\lambda}{\alpha}$ for all $\lambda \in \Zp^d$.
\end{itemize}
\end{thm}
We call these $C_\alpha(f)$ the \emph{Mahler coefficients} of $f$. There is an explicit formula for $C_\alpha(f)$ in terms of the values that $f$ takes:
\[ C_\alpha(f) = (\Delta^\alpha f)(0) := \sum_{\beta \leq \alpha} (-1)^{\alpha - \beta} \binom{\alpha}{\beta} f(\beta)\]
where by convention $\beta \leq \alpha$ means that $\beta_i \leq \alpha_i$ for all $i=1,\ldots, d$.

\subsection{The action of $C^\infty$ on $kG$}
\label{DivPow}
Let $C^\infty = C^\infty(G,k)$ denote the set of locally constant functions $f : G \to k$. Because we always view the base field $k$ as a discrete topological space and because the group $G$ is profinite, $C^\infty$ is also the set of continuous functions $C(G,k)$. 

We showed in \cite[\S 2]{Ard2011} that $C^\infty$ is naturally a commutative Hopf algebra over $k$ and that $kG$ is a $C^\infty$-module algebra. Let  $\rho : C^\infty \to \End_k(kG)$ be the associated $k$-algebra homomorphism; the proof \cite[Proposition 2.5]{Ard2011} shows that the action of $C^\infty$ on $kG$ has the following properties:
\begin{itemize}
\item if $U$ is an open subgroup of $G$ with characteristic function $\delta_U \in C^\infty$ and $\{1, g_2, \ldots, g_m\}$ is a complete set of right coset representatives for $U$ in $G$, then $\rho(\delta_U)$ is the projection of $kG$ onto $kU$ along the decomposition 
\[kG = kU \oplus \bigoplus_{i=2}^m kUg_i,\]
\item $ f\cdot g = f(g) g$ for all $f \in C^\infty$ and $g \in G$. 
\end{itemize}
The $k$-vector space spanned by the characteristic functions $\delta_{Ug}$ of right cosets of $U$ in $G$ can be identified with the subalgebra $C^{\infty \hspace{1mm} U}$ of right $U$-invariants in $C^\infty$:
\[C^{\infty \hspace{1mm} U} = \{f \in C^\infty : f(Ug) = f(g) \quad\mbox{for all}\quad g \in G\}.\]
It follows from \cite[Lemma 2.6(a), Proposition 2.8]{Ard2011} that $U$ controls $I$ if and only if $I$ is a $C^{\infty \hspace{1mm} U}$-submodule of $kG$ via $\rho$.

\subsection{Quantized divided powers}\label{qdel}
Since $k$ is a field of characteristic $p$, there is a unique ``reduction mod $p$" ring homomorphism $\iota_k : \Zp \to k$. We will frequently abuse notation and simply write $\lambda = \iota_k(\lambda)$ for any $\lambda \in \Zp$. The following endomorphisms of $kG$ will play a crucial role in what follows.

\begin{defn} Let $\qdel{\alpha} := \rho\left( \iota_k \circ \binom{ - }{\alpha} \circ \theta_{\mathbf{g}} \right) \in \End_k(kG)$ for all $\alpha \in \mathbb{N}^d$.\end{defn}

The notation is designed to suggest a ``divided power differential operator" and is supported by the following computation. Recall the notion of \emph{bounded $k$-linear maps} from $\S \ref{BddHoms}$. 

\begin{thm} Let $\alpha \in \mathbb{N}^d$. Then 
\be \item $\qdel{\alpha}( \mathbf{g}^\lambda ) = \binom{\lambda}{\alpha} \mathbf{g}^\lambda$ for all $\lambda \in \Zp^d$.
\item $w\left( \qdel{\alpha}( \mathbf{b}^\beta ) - \binom{\beta}{\alpha} \mathbf{b}^{\beta - \alpha} \right) > w\left(\binom{\beta}{\alpha} \mathbf{b}^{\beta - \alpha}\right)$ for all $\beta \in \mathbb{N}^d$. 
\item The operator $\qdel{\alpha} : kG \to kG$ is bounded in the sense of $\S \ref{BddHoms}$.
\item $\deg \qdel{\alpha} = - \wt{\alpha}.$
\ee\end{thm}
\begin{proof} (a) Since $\rho(f)(g) = f \cdot g = f(g) g$ for all $f \in C^\infty$ and $g \in G$, we have
\[\qdel{\alpha}( \mathbf{g}^\lambda ) = \left(\iota_k \circ \binom{-}{\alpha}\circ \theta_{\mathbf{g}}\right)(\mathbf{g}^\lambda) \mathbf{g}^\lambda = \binom{\lambda}{\alpha} \mathbf{g}^\lambda.\]
(b) Suppose first that $d = 1$ and write $g = g_1$ and $b = b_1 = g - 1$. Then 
\[ \partial_g^{(\alpha)}( g^\lambda ) = \frac{g^\alpha}{\alpha!}\frac{d^\alpha}{db^\alpha} ( g^\lambda )\]
for all $\lambda \in \Zp^d$. Since the group elements $g^\lambda$ span a dense subset of $kG = k[[b]]$ and since both $\partial_g^{(\alpha)}$ and the differential operator $\frac{g^\alpha}{\alpha!}\frac{d^\alpha}{db^\alpha}$ are continuous, we see that
\[ \partial_g^{(\alpha)} = \frac{g^\alpha}{\alpha!}\frac{d^\alpha}{db^\alpha}\]
and in particular,
\[\partial_g^{(\alpha)}(b^\beta) = g^\alpha \binom{\beta}{\alpha} b^{\beta - \alpha}\]
in this case. Returning to the general case and applying part (a), we have a factorization
\[\begin{array}{lll} \qdel{\alpha}( \mathbf{b}^\beta ) &=& \sum_{\gamma \leq \beta} (-1)^\gamma \binom{\beta}{\gamma} \binom{\gamma}{\alpha} \mathbf{g}^\gamma  \\
&=& \sum\limits_{\gamma_1=0}^{\beta_1} \cdots \sum\limits_{\gamma_d=0}^{\beta_d} (-1)^{\gamma_1 + \ldots + \gamma_d} \binom{\beta_1}{\gamma_1} \cdots \binom{\beta_d}{\gamma_d} \binom{\gamma_1}{\alpha_1}\cdots \binom{\gamma_d}{\alpha_d} g_1^{\gamma_1}\cdots g_d^{\gamma_d}  \\
&=& \prod\limits_{i=1}^d \sum\limits_{\gamma_i=0}^{\beta_i} (-1)^{\gamma_i} \binom{\beta_i}{\gamma_i} \binom{\gamma_i}{\alpha_i} g_i^{\gamma_i}  \\
&=& \prod\limits_{i=1}^d g_i^{\alpha_i} \binom{\beta_i}{\alpha_i} b_i^{\beta_i - \alpha_i}  \end{array}\]
by the one-dimensional case applied to each procyclic subgroup $\langle g_i\rangle$ of $G$. Thus
\begin{equation}\label{CompQdel} \qdel{\alpha}( \mathbf{b}^\beta )= \binom{\beta}{\alpha} \prod\limits_{i=1}^d (1 + b_i)^{\alpha_i} b_i^{\beta_i - \alpha_i}\quad\mbox{for all}\quad \alpha,\beta \in \mathbb{N}^d.\end{equation}
Using Lemma \ref{valkG}(c), we see that the leading term of this expression with respect to the valuation $w$ is simply $\binom{\beta}{\alpha} \mathbf{b}^{\beta - \alpha}$ as claimed. Note that in particular it is zero whenever $\alpha_i > \beta_i$ for some $i$.

(c), (d) Using part (b) and Corollary \ref{valkG}(b) shows that
\[w(\qdel{\alpha}(x))  \geq w(x) - \langle \alpha, \omega(\mathbf{g})\rangle\]
whenever $x$ is a finite $k$-linear combination of monomials $\mathbf{b}^\beta$. Since elements of this form span a dense subalgebra inside $kG$ by Corollary \ref{valkG}(a), the inequality holds for all $x \in kG$. So $\qdel{\alpha}$ is bounded with 
\[\deg \qdel{\alpha} \geq - \langle \alpha, \omega(\mathbf{g})\rangle.\] 
On the other hand, taking $\beta = \alpha$ in the formula $(\ref{CompQdel})$ above shows that
\[ \qdel{\alpha}(\mathbf{b}^\alpha) = 1\]
which forces $\deg \qdel{\alpha} \leq -\langle \alpha,\omega(\mathbf{g})\rangle$.
\end{proof}

So the graded endomorphism of $\gr kG = k[X_1,\ldots,X_d]$ induced by $\qdel{\alpha}$ is the divided power $\frac{1}{\alpha!} \frac{\partial^{\alpha}}{\partial X^\alpha}$, which suggests that we should think of $\qdel{\alpha}$ as being a ``quantized divided power". 
\begin{cor} $\rho(C^\infty) \subseteq \mathcal{B}(kG)$.\end{cor}
\begin{proof} Since $k$ carries the discrete topology, the condition $C_\alpha(f) \to 0$ on the Mahler coefficients of $f \in C^\infty$ means that $C_\alpha(f) = 0$ for all sufficiently large $\alpha$. Thus Mahler's Theorem \ref{Mahler} implies that $C^\infty$ is spanned over $k$ by the binomial coefficients $\iota_k \circ \binom{-}{\alpha} \circ \theta_{\mathbf{g}}$. But $\qdel{\alpha} = \rho \left(\iota_k \circ \binom{-}{\alpha} \circ \theta_{\mathbf{g}}\right) \in \mathcal{B}(kG)$ by the Theorem above.
\end{proof}
We will view these quantized divided powers as a particularly nice ``orthonormal basis" for the space $\mathcal{B}(kG)$ of bounded endomorphisms of $kG$.

\subsection{Extending automorphisms to $kG$}
\label{Extend}
\begin{lem} Let $\varphi : G \to H$ be a continuous group homomorphism to another complete $p$-valued group $H$ of finite rank. Then the linear extension
\[k[\varphi] : k[G] \to k[H]\]
is continuous with respect to the topologies on these group rings defined by the valuations $w$ given in $\S\ref{valkG}$ and therefore extends to a continuous map $\varphi : kG\to kH$.\end{lem}
\begin{proof} By Lemma \ref{valkG}(d), the topology on $k[G]$ defined by the valuation $w$ has the augmentation ideals $(G_\lambda - 1)k[G]$, $\lambda \in \mathbb{R}$, as a fundamental system of neighbourhoods of $0$. If $f(G_\mu) \subseteq H_\lambda$ then $k[f]$ sends $(G_\mu - 1)k[G]$ into $(H_\lambda - 1)k[H]$ which shows that $k[f]$ is continuous. The second statement is clear.
\end{proof}
The assumption of continuity on $\varphi$ is actually redundant because any abstract group homomorphism from a finitely generated pro-$p$ group to another profinite group is automatically continuous by \cite[Corollary 1.21(i)]{DDMS}.

\begin{prop} Let $\varphi \in \Aut^\omega(G)$ and let $\psi(g) = \varphi(g)g^{-1}$. Then 
\[ \varphi(g) = \sum_{\alpha \in \mathbb{N}^d} (\Delta^\alpha \psi \theta^{-1}_{\mathbf{g}})(0) \qdel{\alpha}(g) \]
for all $g \in G$, the right hand side converging in the topology of $kG$.
\end{prop}
\begin{proof} The map $\psi \circ \theta_{\mathbf{g}}^{-1} : \Zp^d \to G$ is continuous, so we can view it as a continuous map from $\Zp^d$ to the complete $\Zp$-module $kG$.  By Mahler's Theorem \ref{Mahler},
\[ \psi(\mathbf{g}^\lambda) = \sum_{\alpha \in \mathbb{N}^d} (\Delta^\alpha \psi \theta^{-1}_{\mathbf{g}})(0) \binom{\lambda}{\alpha}\]
for all $\lambda \in \Zp^d$. But $\qdel{\alpha}(\mathbf{g}^\lambda) = \binom{\lambda}{\alpha} \mathbf{g}^\lambda$ by Theorem \ref{qdel}(a), so multiplying both sides of this convergent sum on the right by $\mathbf{g}^\lambda$ gives the result.
\end{proof}

\begin{defn} Let $\varphi \in \Aut^\omega(G)$ and let $\alpha \in \mathbb{N}^d$. The \emph{$\alpha$-Mahler coefficient} of $\varphi$ is the element 
\[\langle \varphi, \qdel{\alpha}\rangle := (\Delta^\alpha \psi \theta^{-1}_{\mathbf{g}})(0) \in kG\] 
where $\psi : G \to G$ is the function $g \mapsto \varphi(g)g^{-1}$.
\end{defn}

\begin{cor} Suppose that the Mahler coefficients $\langle \varphi, \qdel{\alpha}\rangle$ of $\varphi$ satisfy
\[w(\langle \varphi, \qdel{\alpha}\rangle) - \langle \alpha, \omega(\mathbf{g})\rangle \to \infty\quad\mbox{as}\quad\alpha \to \infty.\]
Then the extension of $\varphi$ to $kG$ is a bounded linear endomorphism of $kG$ and 
\[\varphi = \sum_{\alpha \in \mathbb{N}^d} \langle \varphi, \qdel{\alpha} \rangle \qdel{\alpha}\]
 inside $\mathcal{B}(kG)$.
\end{cor}
\begin{proof} Let us identify $kG$ with the subring of $\mathcal{B}(kG)$ consisting of left multiplications by elements of $kG$; clearly then $\deg(x) = w(x)$ for all $x \in kG$. Now $\deg \qdel{\alpha} = -\langle \alpha, \omega(\mathbf{g})\rangle$ for all $\alpha$ by Theorem \ref{qdel}(d) and
\[ \deg( \langle \varphi, \qdel{\alpha}\rangle \qdel{\alpha} ) \geq w(\langle \varphi, \qdel{\alpha}\rangle) - \langle \alpha,\omega(\mathbf{g})\rangle \to \infty\]
as $\alpha\to \infty$ by assumption, so the infinite sum $\sum_{\alpha\in\mathbb{N}^d} \langle \varphi, \qdel{\alpha}\rangle \qdel{\alpha}$ converges to an operator in $\mathcal{B}(kG)$ by Lemma \ref{BddHoms}. Since this operator agrees with $\varphi : kG \to kG$ on the dense subspace $k[G]$ of $kG$ by the Proposition, the two operators are equal everywhere and the result follows.
\end{proof}

\subsection{Calculating Mahler coefficients}\label{CalcMahler}

In general it is not completely straightforward to compute the Mahler coefficient $\langle \varphi, \qdel{\alpha}\rangle$ of the extension $\varphi : kG \to kG$; but in some cases we do get a nice result.

\begin{lem} Let $\varphi \in \Aut^\omega_Z(G)$ and let $\psi(g) = \varphi(g)g^{-1}$. Then
\begin{equation}\label{phialpha} \langle \varphi, \qdel{\alpha}\rangle = (\psi(g_1) - 1)^{\alpha_1}\cdots (\psi(g_d) - 1)^{\alpha_d}.\end{equation}
for all $\alpha \in \mathbb{N}^d$. \end{lem}
\begin{proof} Because $\varphi$ is trivial mod centre by assumption, $\psi : G \to Z$ is a group homomorphism:
\[\psi(gh) = \varphi(gh)(gh)^{-1} = \varphi(g) \left( \varphi(h) h^{-1} \right) g^{-1} = \psi(g)\psi(h)\]
for all $g, h \in G$. Hence $\psi(\mathbf{g}^\beta) = \prod\limits_{i=1}^d \psi(g_i)^{\beta_i}$ for all $\beta \in \mathbb{N}^d$. Now
\[\begin{array}{lll} \langle \varphi, \qdel{\alpha}\rangle &=& \sum\limits_{\beta \in \mathbb{N}^d} (-1)^{\alpha - \beta} \binom{\beta}{\alpha} \psi(\mathbf{g}^\beta) = \sum\limits_{\beta \in \mathbb{N}^d} (-1)^{\alpha - \beta} \binom{\beta}{\alpha} \prod\limits_{i=1}^d \psi(g_i)^{\beta_i} \\
&=& \prod\limits_{i=1}^d\sum\limits_{\beta_i = 0}^{\alpha_i} (-1)^{\alpha_i - \beta_i} \binom{\beta_i}{\alpha_i} \psi(g_i)^{\beta_i} = \prod\limits_{i=1}^d (\psi(g_i) - 1)^{\alpha_i}
\end{array}\]
by the binomial theorem.\end{proof}
In fact it can be shown that the Lemma holds for an automorphism $\varphi \in \Aut^\omega(G)$ if and only if $\varphi$ is trivial mod centre.
\begin{cor} The extension of any $\varphi \in \Aut^\omega_Z(G)$ to $kG$ is bounded.
\end{cor}
\begin{proof} By Lemma \ref{CalcMahler} we have
\[w(\langle \varphi, \qdel{\alpha}\rangle) - \langle \alpha, \omega(\mathbf{g})\rangle = \sum\limits_{i=1}^d \alpha_i (\omega(\varphi(g_i)g_i^{-1}) - \omega(g_i)) \geq \deg_\omega(\varphi) |\alpha|\]
which tends to $\infty$ as $\alpha \to \infty$ because $\deg_\omega(\varphi) > 1/(p-1) > 0$ by assumption. Now apply Corollary \ref{Extend}.
\end{proof}

\section{Control theorem for faithful prime ideals}
\label{MainSec}
\subsection{Notation}
We now start working towards the proof of Theorem \ref{MainB}, which is given in $\S \ref{PfCntrlThm}$. From now on, we will fix the complete $p$-valued group $G$ of finite rank with $p$-valuation $\omega$ and centre $Z$. We assume that $\omega$ takes values in $\frac{1}{e}\mathbb{Z} \cup \{\infty\}$ --- see $\S \ref{RatVal}$. We also fix the prime ideal $P$ of $kG$, and assume that $P$ is not the maximal ideal $\mathfrak{m}$ of $kG$.

Let $Q$ be the Goldie classical ring of quotients of the prime Noetherian algebra $R = kG/P$, and let $\tau : kG \to Q$ be the composition of the surjection $kG \twoheadrightarrow R$ with the inclusion $R \hookrightarrow Q$. We will denote by $F$ the centre of $Q$; this is a commutative field which contains $\tau(kZ)$.

\subsection{Finding a good filtration on $kG/P$}
\label{FiltQ}
Recall from Lemma \ref{valkG} that $kG$ carries a filtration $w$ which is independent of any choice of ordered basis for $G$. Since $\omega$ takes values in $\frac{1}{e}\mathbb{Z} \cup \{\infty\}$, the function $x \mapsto ew(x)$ actually takes integer values on non-zero elements in $kG$. Note that this function is also a filtration on $kG$.

Let $\overline{ew} : R \to \mathbb{Z} \cup \{\infty\}$ be the quotient filtration on $R$ defined by 
\[\overline{ew}(\tau(x)) = \sup\limits_{y \in P} ew(x + y).\]
Let $R_n = \{ x + P \in R : \overline{ew}(x+ P) \geq n\}$ be the corresponding subgroups of $R$.
\begin{lem} The filtration $\{R_n : n \in \mathbb{Z}\}$ is Zariskian, the associated graded ring $\gr R$ is commutative, $R_0 / R_1$ is a field and  $\gr R$ is infinite dimensional over $R_0/R_1$.
\end{lem}
\begin{proof} Choose an ordered basis  $\{g_1,\ldots, g_d\}$ for $G$ as in $\S \ref{RatVal}$. Since $\gr kG \cong k[X_1,\ldots, X_d]$ and the filtration $w$ is complete by Lemma \ref{valkG}, $ew$ is a Zariskian filtration on $kG$ by \cite[Proposition II.2.2.1]{LVO}. This property is inherited by the factor ring $R = kG/P$, so $\{R_n : n \in \mathbb{Z}\}$ is a Zariskian iltration on $R$ and moreover $\gr R = \gr kG / \gr P$ is commutative. 

Using Corollary \ref{valkG}(b) we see that $kG = k + \sum_{i=1}^d b_i kG$, so $R = k + \sum_{i=1}^d \tau(b_i) R$. But $\overline{ew}(\tau(b_i)) \geq ew(b_i) = e \omega(g_i) > 0$ for all $i$ by Lemma \ref{valkG}(b), so $\tau(b_i) \in R_1$ for all $i$ since $\overline{ew}$ takes integer values. Hence $R = k + R_1$ and $R_0 / R_1 \cong k$ is a field. 

Finally, if $\gr R$ is finite dimensional over $k$ then so is $R$ --- but then $P = \mathfrak{m}$ since $P$ is prime and $\mathfrak{m}$ is the unique maximal ideal of $kG$. This is not the case by assumption.
\end{proof}

\subsection{Finding a good valuation on $Q$}\label{RescVal}
We will always consider $kG$ as a filtered ring with the filtration $w$ given by Lemma \ref{valkG}. The heart of our proof is concerned with manipulations involving bounded linear maps $kG \to Q$, and it will be essential to know that the natural algebra homomorphism $\tau : kG \to Q$ is bounded.

\begin{thm} There exists a filtration $v : Q \to \mathbb{Z} \cup \{\infty\}$ such that 
\begin{enumerate}[{(}a{)}] 
\item $v(\tau(x)) \geq 0$ for all $x \in kG$, 
\item the restriction of $v$ to $F = Z(Q)$ is a valuation,
\item $v(\tau(x)) \geq w(x)$ for all $x \in kZ$, and
\item the map $\tau : (kG, w) \to (Q, v)$ is bounded.
\end{enumerate}
\end{thm}
\begin{proof} By Lemma \ref{FiltQ} and Theorem \ref{MainC}, we can find an integer valued filtration $v_0$ on $Q$ such that the natural inclusion $(R,\overline{ew}) \to (Q,v_0)$ is continuous and (a) and (b) are satisfied for $v_0$. Hence $\tau : (kG,w) \to (Q,v_0)$ is also continuous, but unfortunately, not every continuous map is bounded. We will remedy this problem by rescaling $v_0$.

The inclusion of $G$ into the group of units of $kG$ is continuous; since $\tau$ is continuous, the subgroup 
\[ U := \{ g \in G : v_0(\tau(g) - 1) \geq 1\}\]
is open in $G$. By Lemma \ref{OrdBas} we can choose an ordered basis $\{g_1,\ldots,g_d\}$ for $G$ such that $\{h_1,\ldots, h_d\}$ is an ordered basis for $U$ where $h_i = g_i^{p^{n_i}}$ for some integers $n_i$. Let $M$ be any integer greater than each of the $\omega(h_i)$ and define 
\[v := Mv_0 : Q \to \mathbb{Z} \cup \{\infty\}.\]  
Then $v$ is a filtration on $Q$ which satisfies (a) and (b). If $z \in Z$ then $z^{p^a} \in U$ for some integer $a$ because $U$ is open in $G$, so 
\[v_0(\tau(z) - 1) = \frac{v_0\left((\tau(z) - 1)^{p^a}\right)}{p^a} =\frac{v_0(\tau(z^{p^a}) - 1)}{p^a} \geq \frac{1}{p^a}\]
because $v_0$ is a valuation on $F \supseteq \tau(kZ)$. Since $v_0$ takes integer values, we see that $Z$ is contained in $U$. 

Write $\mathbf{c}^\beta = (h_1-1)^{\beta_1}\cdots(h_d-1)^{\beta_d} \in kU$ for any $\beta \in \mathbb{N}^d$. Our choice of $U$ forces $v_0(\tau(\mathbf{c}^\beta)) \geq |\beta|$ for all $\beta \in \mathbb{N}^d$, so 
\[ v(\tau(\mathbf{c}^\beta)) \geq M |\beta| \geq \langle \beta, \omega(\mathbf{h})\rangle = w(\mathbf{c}^\beta)\quad\mbox{for all}\quad \beta \in \mathbb{N}^d\]
by Lemma \ref{valkG}(c). Let $x = \sum\limits_{\beta\in\mathbb{N}^d} \lambda_\beta \mathbf{c}^\beta \in kU$, then $w(x) = \inf \{ w(\mathbf{c}^\beta) : \lambda_\beta \neq 0 \}$ by Corollary \ref{valkG}(b), so
\begin{equation}\label{vtau}  v(\tau(x))  \geq w(x) \mb{for all} x \in kU\end{equation}
and in particular $v(\tau(x)) \geq w(x)$ for all $x \in kZ$ since $Z \subseteq U$. It remains to show that $\tau : (kG, w) \to (Q, v)$ is bounded.

Define $S = \{ \alpha \in \mathbb{N}^d : 0 \leq \alpha_i < p^{n_i}$ for all $i = 1,\ldots, d\}$. Since $\gr kG$ is a free $\gr kU$-module with basis $\{\gr \mathbf{b}^\alpha : \alpha \in S\}$, \cite[Th\'eor\`eme I.2.3.17]{Laz1965} tells us that every element $x \in kG$ can be written in the form $x = \sum_{\alpha \in S} x_\alpha \mathbf{b}^\alpha $ for some $x_\alpha \in kU$, and moreover
\[w(x) = \inf \{ w(x_\alpha) + \wt{\alpha} : \alpha \in S\}.\]
Because $v_0(\tau(\mathbf{b}^\alpha)) \geq 0$ for all $\alpha \in S$ by construction, applying $(\ref{vtau})$ shows that
\[\begin{array}{llllll} v(\tau(x)) &=& v\left( \sum\limits_{\alpha \in S} \tau(x_\alpha) \tau(\mathbf{b}^\alpha) \right) &\geq& \inf \{ v(\tau(x_\alpha)) : \alpha \in S\} \\ &\geq & \inf \{ w(x_\alpha) : \alpha \in S\} &\geq& w(x) - \sup \{ w(\mathbf{b}^\alpha) : \alpha \in S \}&\end{array}\]
for all $x = \sum\limits_{\alpha\in S}x_\alpha \mathbf{b}^\alpha \in kG$. Therefore $\tau : (kG, w) \to (Q, v)$ is bounded, and 
$\deg(\tau) \geq - \sup\{w(\mathbf{b}^\alpha) : \alpha \in S\}.$ \end{proof}

From now on we fix a filtration $v$ on $Q$ satisfying the conclusion of Theorem \ref{RescVal}.

\subsection{The number $\lambda$}
\label{lambda}
Recall from $\S \ref{SatIsNice}$ and Proposition \ref{AutModZ}(c) that for any $\varphi \in \Aut^\omega_Z(G)$ we have defined a group homomorphism $z(\tilde{\varphi}) : \Sat(G) \to \Sat(G)$, and that the image of $z(\tilde{\varphi})$ is contained in $Z(\Sat(G))$ by Propositions \ref{AutModZ} and \ref{SatIsNice}. \emph{A priori} this image is not even contained in $G$. 

\begin{lem} There exists an integer $r_1$ such that for any $\varphi \in \Aut^\omega_Z(G)$ and any $r \geq r_1$, the image of $z(\tilde{\varphi}^{p^r})$ is contained in $Z$ and \[v\left(\tau(z(\tilde{\varphi}^{p^r})(g) - 1)\right) \geq 1 \quad\mbox{for all}\quad g \in G.\]
\end{lem}
\begin{proof} It is easy to see that $Z(\Sat(G)) = \Sat(Z)$. Since $Z$ has finite rank, $Z$ is open in $Z(\Sat(G))$ by \cite[Theorem IV.3.4.1]{Laz1965}, so 
\[Z(\Sat(G))^{p^{r_1}} \subseteq Z\]
for some integer $r_1$. But $z(\tilde{\varphi}^{p^r})(g) = z(\tilde{\varphi})(g)^{p^r}$ by the definition of $z(\tilde{\varphi})$, so the image of $z(\tilde{\varphi}^{p^r})$ is contained in $Z$ whenever $r \geq r_1$. Now $v (\tau( x - 1 )) > 0$ for all $x \in Z$ by Theorem \ref{RescVal}(c); since $v$ takes integer values, actually $v(\tau(x - 1)) \geq 1$ for all $x \in Z$. The result follows.
\end{proof}

We now fix a non-trivial automorphism $\varphi \in \Aut^\omega_Z(G)$ such that $z := z(\tilde{\varphi})$ sends $\Sat(G)$ into $Z$, and such that $v(\tau(z(g) - 1)) \geq 1$ for all $g \in G$. Such an automorphism always exists because $\Aut^\omega(G)$ is torsion-free by Corollary \ref{Sat}(b). We define
\[\lambda := \inf\limits_{g \in G} v(\tau(z(g) - 1)) \geq 1.\]

\subsection{Using the fact that $P$ is faithful}\label{UseFaith}
We now assume that $P$ is faithful, and crucially use this fact in the proof of the following

\begin{prop} $\lambda$ is finite and there exists $1 \neq g \in G$ such that $\lambda = v(\tau(z(g) - 1))$. \end{prop}
\begin{proof} Suppose for a contradiction that $\lambda = \infty$. Because $v_{|F}$ is a valuation by Theorem \ref{RescVal}(b) and $\tau(z(g)) \in F$ for all $g \in G$, we see that  $z(g) - 1 \in P$ for all $g \in G$. Since $P$ is faithful, this implies that $z(g) = 1$ for all $g \in G$. But $z = (\log \tilde{\varphi}_\ast)^\ast$ by definition, so $\log \tilde{\varphi}_\ast$ must send everything in $\log(\tilde{G})$ to zero, which forces $\tilde{\varphi}$ and hence $\varphi$ to be the trivial automorphism, a contradiction.

By Lemma \ref{lambda}, $g \mapsto v(\tau(z(g) - 1))$ is a function $G \to [1,\infty]$. If we give $[1,\infty]$ the topology where the open neighbourhoods of $\infty$ are the sets $(\nu, \infty]$ for all $\nu \geq 1$, then this function is continuous and therefore attains its minimum value at some $g \in G$ because $G$ is compact.
\end{proof}

\subsection{The subgroup $H$}\label{SubH}
Consider the $\lambda$-th piece $Q_\lambda / Q_{\lambda^+}$ of $\gr Q$ as an abelian group, and define
\[\sigma : G \to Q_\lambda / Q_{\lambda^+} \quad\mbox{by} \quad \sigma(g) = \tau(z(g) - 1) + Q_{\lambda^+}.\]
We now construct the subgroup $H$ which features in the statement of Theorem \ref{MainB}.
\begin{lem} The map $\sigma$ is a group homomorphism, and $H := \ker \sigma$ is a proper subgroup of $G$ which contains the Frattini subgroup $\Phi(G)$ of $G$.
\end{lem}
\begin{proof} Since $z$ is a group homomorphism by Proposition \ref{AutModZ}(c), 
\[z(gh) - 1 = (z(g)-1) + (z(h)-1) + (z(g) - 1)(z(h) - 1)\quad\mbox{for any}\quad  g,h \in G\]
and $v(\tau((z(g)-1)(z(h)-1))) \geq 2 \lambda > \lambda$ because $\lambda \geq 1$. So $\sigma$ is a group homomorphism and $\sigma(g) \neq 0$ for some $g\in G$ by Proposition \ref{UseFaith}. Finally $Q_\lambda / Q_{\lambda^+}$ is an abelian group of exponent $p$ so $H$ must contain $\Phi(G)$.
\end{proof}
We choose an ordered basis $\{g_1,\ldots, g_d\}$ for $G$ such that $\{g_1^{p^{n_1}}, \ldots, g_d^{p^{n_d}}\}$ is an ordered basis for $H$ for some increasing sequence of integers $n_i$, using Lemma \ref{OrdBas}. Let us reorder these bases in such a way that the sequence of integers $n_i$ becomes decreasing; we also know that $n_i \leq 1$ for all $i$ because $G^p \subseteq H$ by the Lemma. Let $m$ be the greatest integer such that $n_m = 1$; since $H$ is a proper subgroup by the Lemma, $2 \leq m \leq d$ so
\begin{itemize}
\item $\{g_1,\ldots,g_d\}$ is an ordered basis for $G$, 
\item $\{g_1^p,\ldots, g_m^p, g_{m+1}, \cdots, g_d\}$ is an ordered basis for $H$.
\end{itemize}
So equivalently, $m = \log_p |G/H|$. We fix this ordered basis of $G$ until the end of $\S \ref{MainSec}$.

\subsection{Expansions in $\mathcal{B}(kG, Q)$}
\label{ModPExp}
By Corollary \ref{Extend}, we know that inside $\mathcal{B}(kG)$ 
\[\varphi = \sum_{\alpha \in \mathbb{N}^d} \langle \varphi, \qdel{\alpha} \rangle \qdel{\alpha}.\]
The map $\tau : kG \to Q$ is bounded by Theorem \ref{RescVal}, so the sequence
\[ \tau \varphi - \sum_{|\alpha| \leq k} \tau \left(\langle \varphi, \qdel{\alpha} \rangle\right) \cdot \tau \qdel{\alpha}\]
converges to zero inside $\mathcal{B}(kG,Q)$ and we may write
\[\tau \varphi^{p^r} = \sum_{\alpha \in \mathbb{N}^d} \tau \left(\langle \varphi^{p^r}, \qdel{\alpha} \rangle\right) \cdot \tau \qdel{\alpha} \quad\mbox{for all}\quad r \geq 0\]
inside $\mathcal{B}(kG,Q)$. For each $i = 1,\ldots, d$, define the element
\[y_i := \tau(z(g_i) - 1) \in Q,\]
and for each $\alpha \in \mathbb{N}^d$, write $\mathbf{y}^\alpha := y_1^{\alpha_1} \cdots y_d^{\alpha_d} \in Q$. It turns out that the Mahler coefficient $\tau \left(\langle \varphi^{p^r}, \qdel{\alpha} \rangle\right)$ is asymptotically very close to the power $\mathbf{y}^{\alpha p^r}$ for each $\alpha \in \mathbb{N}^d$. More precisely, we have the following

\begin{prop} There exists an integer $r_2 \geq r_1$ such that 
\[ v\left( \tau \left(\langle \varphi^{p^r}, \qdel{\alpha} \rangle\right) - \mathbf{y}^{\alpha p^r} \right) \geq p^{2r - r_1} + \lambda p^r (|\alpha| - 1)\]
for all $0 \neq \alpha \in \mathbb{N}^d$ and all $r \geq r_2$.
\end{prop}
\begin{proof} Recall the integer $r_1$ from Lemma \ref{lambda}. Fix $r \geq r_1$, and define for each $i=1,\ldots, d$ the ``error terms"
\[ b_{ir} := \tau(\varphi^{p^r}(g_i) g_i^{-1} - 1) - y_i^{p^r}  \in Q.\]
Using Lemma \ref{CalcMahler}, we can then rewrite the image of the Mahler coefficient $\langle \varphi^{p^r}, \qdel{\alpha}\rangle$ in $Q$ as follows:
\begin{equation}\label{PhiExp}\begin{array}{lll} \tau \left(\langle \varphi^{p^r}, \qdel{\alpha}\rangle\right) &=& \tau\left((\varphi^{p^r}(g_1)g_1^{-1} - 1)^{\alpha_1}\cdots (\varphi^{p^r}(g_d)g_d^{-1} - 1)^{\alpha_d}\right)  \\
&=& (y_1^{p^r} + b_{1r})^{\alpha_1} \cdots (y_d^{p^r} + b_{dr})^{\alpha_d}\end{array}.\end{equation}
On the other hand, for any $g \in G$ there exists some $\epsilon_r(g) \in \Sat(G)$ such that
\[ \varphi^{p^r}(g) g^{-1} = z(g)^{p^r} \epsilon_r(g)^{p^{2r}}\]
by Proposition \ref{AutModZ}(a). Since $z(g) \in Z$ and $\varphi$ is trivial mod centre, we see that $\epsilon_r(g)^{p^{2r}} \in Z$ and therefore $\epsilon_r(g) \in \Sat(Z)$. But $\Sat(Z)^{p^{r_1}} \subseteq Z$ by the definition of $r_1$ so $\epsilon_r'(g) := \epsilon_r(g)^{p^{r_1}} \in Z$ always and
\[\varphi^{p^r}(g) g^{-1} = z(g)^{p^r} \epsilon_r'(g)^{p^{2r-r_1}}\]
whenever $r \geq r_1$, say. Therefore
\[\begin{array}{lll} v(b_{ir}) &=& v\tau\left(\varphi^{p^r}(g_i)g_i^{-1} - 1 - (z(g_i) - 1)^{p^r}\right) = v\tau\left( z(g_i)^{p^r}\left(\epsilon_r'(g_i)^{p^{2r-r_1}} - 1\right)\right) \\
&\geq& v\tau \left( (\epsilon_r'(g_i) - 1)^{p^{2r - r_1}}\right) \geq p^{2r-r_1} \mb{whenever} r \geq r_1\end{array}\]
for each $i$, because $v\tau(x - 1) \geq 1$ for all $x \in Z$. 

Choose $r_2 \geq r_1$ such that $p^{2r - r_1} > \lambda p^r$ whenever $r \geq r_2$; expanding $(\ref{PhiExp})$ shows that if $r \geq r_2$ then $\tau (\langle \varphi^{p^r}, \qdel{\alpha}\rangle) - \mathbf{y}^{\alpha p^r}$ is a linear combination of products of length $|\alpha|$, where each product contains at least one $b_{ir}$ with the $b_{ir}$ of greater value than $\lambda p^r$ by the choice of $r_2$. The result follows.
\end{proof}

\subsection{The values of certain linear forms}\label{Vdet}
Proposition \ref{ModPExp} tells us that the terms $\mathbf{y}^{p^r \alpha} \tau \qdel{\alpha}$ are dominant in the Mahler expansion of $\tau \varphi^{p^r}$. We will now study the growth rates of these terms more closely.
\begin{lem} \begin{enumerate}[{(}a{)}]
\item For any $\mu_1,\ldots,\mu_m \in \Fp$, not all zero, and any $r \geq 0$,
\[v\left( \sum_{i=1}^m \mu_i y_i^{p^r} \right) = p^r \lambda.\]
\item $\lambda = v(y_1) = \cdots = v(y_m) < v(y_\ell)$ for all $\ell > m$.
\end{enumerate} \end{lem}
\begin{proof}
(a) Choose $\alpha_1,\ldots,\alpha_m \in \mathbb{N}$ such that $\mu_i$ is the image of $\alpha_i$ in $\Fp$ for each $i$, and define $g := g_1^{\alpha_1}\cdots g_m^{\alpha_m} \in G$. Because some $\mu_i$ is non-zero, $g \notin H$, so $\sigma(g) \neq 0$. Now $\sigma(g_i) = y_i + Q_{\lambda^+}$ and $\sigma$ is a group homomorphism by Lemma \ref{SubH}, so
\[\sigma(g) = \sum_{i=1}^m \alpha_iy_i + Q_{\lambda^+} = \sum_{i=1}^m \mu_iy_i + Q_{\lambda^+} \neq 0.\]
Therefore $v( \sum_{i=1}^m \mu_i y_i) = \lambda$ and
\[v\left( \sum_{i=1}^m \mu_i y_i^{p^r} \right) = v\left( (\sum_{i=1}^m \mu_i y_i)^{p^r}\right) = p^r \lambda\]
because $v_{|F}$ is a valuation by Theorem \ref{RescVal}(b).

(b) Part (a) implies that $v(y_1) = \cdots = v(y_m) = \lambda$. If $\ell > m$ then $g_\ell \in H$ by our choice of ordered basis of $G$, so $\sigma(g_\ell) = \tau(z(g_\ell) - 1) + Q_{\lambda^+} = 0$ and $v(y_\ell)> \lambda.$
\end{proof}

\subsection{The Smith matrix}
Write $\partial_i := \qdel{e_i}$ where $e_i = (0,\ldots,1,\ldots, 0) \in \mathbb{N}^d$ is the $i$-th standard unit vector. By Proposition \ref{ModPExp}, we may write
\[ \tau \varphi^{p^r} - \tau = y_1^{p^r} \tau \partial_1 + y_2^{p^r} \tau \partial_2 + \cdots + y_m^{p^r} \tau \partial_m + \cdots \quad\mbox{whenever} \quad r \geq r_2\]
where the undisplayed terms are growing \emph{faster} with $r$ than the $y_1^{p^r},\cdots, y_m^{p^r}$, which are all growing at the \emph{same} uniform rate $\lambda p^r$ by Lemma \ref{Vdet}. We wish to ``extract" the operators $\tau \partial_i$ from these expansions. To do this, we consider $m$ of these expansions at a time starting with $\tau \varphi^{p^r}$:
\[ \begin{array}{ccccccccccc} \tau \varphi^{p^r} - \tau &=& y_1^{p^r} \tau \partial_1 &+& y_2^{p^r} \tau \partial_2 &+ \cdots +& y_m^{p^r} \tau \partial_m &+& \cdots \\

 \tau \varphi^{p^{r+1}} - \tau &=& y_1^{p^{r+1}} \tau \partial_1 &+& y_2^{p^{r+1}} \tau \partial_2 &+ \cdots +& y_m^{p^{r+1}} \tau \partial_m &+& \cdots
\\ \vdots &=& \vdots & & \vdots &  &\vdots  & &
\\ \tau \varphi^{p^{r+m-1}} - \tau &=& y_1^{p^{r+m-1}} \tau \partial_1 &+& y_2^{p^{r+m-1}} \tau \partial_2 &+ \cdots +& y_m^{p^{r+m-1}} \tau \partial_m &+& \cdots
\end{array}\]
and take an appropriate $F$-linear combination of them. For any $r\geq 0$, define the \emph{Smith matrix} $M_r$ with entries in the field $F$ as follows:
\[M_r :=
\begin{pmatrix}
y_1^{p^r} & y_2^{p^r} & \cdots & y_m^{p^r} \\
y_1^{p^{r+1}} & y_2^{p^{r+1}} & \cdots & y_m^{p^{r+1}} \\
\vdots & \vdots & \cdots& \vdots \\
y_1^{p^{r + m-1}} & y_2^{p^{r + m-1}} & \cdots & y_m^{p^{r + m-1}}
\end{pmatrix}.\]
This matrix has already appeared in \cite[\S 1]{AWZ2}. The expansions considered in $\S \ref{ModPExp}$ can now be rewritten in matrix form as follows:
\begin{equation}\label{MatrixEq} \begin{pmatrix} \tau \varphi^{p^r} - \tau \\ \tau \varphi^{p^{r+1}} - \tau \\ \vdots \\ \tau \varphi^{p^{r+m-1}} - \tau \end{pmatrix} = M_r \cdot \begin{pmatrix} \tau \partial_1 \\ \tau \partial_2 \\ \vdots \\ \tau \partial_m \end{pmatrix} + N_r \mathbf{v} + \begin{pmatrix} \eta_r \\ \eta_{r+1} \\ \vdots \\ \eta_{r+m-1} \end{pmatrix} \end{equation}
where $\mathbf{v}$ is the infinite column vector containing the remaining basis vectors
\[\mathbf{v} := \begin{pmatrix} \tau \partial_{m+1} & \cdots & \tau \partial_d & \tau \qdel{2e_1} & \cdots & \tau \qdel{2e_d} & \cdots \end{pmatrix}^T ,\]
$N_r$ is the $m$-by-infinite matrix
\[N_r := \begin{pmatrix} y_{m+1}^{p^r} & \cdots & y_d^{p^r} & y_1^{2p^r} & \cdots & y_d^{2p^r} & \cdots \\
y_{m+1}^{p^{r+1}} & \cdots & y_d^{p^{r+1}} & y_1^{2p^{r+1}} & \cdots & y_d^{2p^{r+1}} & \cdots \\
\vdots & \vdots & \vdots & \vdots & \vdots & \vdots & \vdots \\
y_{m+1}^{p^{r+m-1}} & \cdots & y_d^{p^{r+m-1}} & y_1^{2p^{r+m-1}} & \cdots & y_d^{2p^{r+m-1}} & \cdots \end{pmatrix} \]
and $\eta_r := \sum\limits_{\alpha\in\mathbb{N}^d}\left(\tau\left( \langle \varphi^{p^r} , \qdel{\alpha}\rangle \right) - \mathbf{y}^{\alpha p^r} \right) \tau \qdel{\alpha} $ is an ``error term".
\subsection{Some linear algebra}
The fact that $v_{|F}$ is a valuation is used crucially in the following 
 
\label{LinAlg}
\begin{prop} The matrix $M_r$ is invertible, and the entries of its inverse satisfy
\[v\left( (M_r^{-1})_{ij}\right) \geq - p^{j + r - 1} \lambda \]
for all $i,j=1,\ldots,m$.
\end{prop}
\begin{proof} Let $[\mu]$ denote the image of $(\mu_1,\ldots,\mu_m) \in \Fp^m \backslash \{0\}$ in the projective space $\mathbb{P}(\Fp^m)$. By \cite[Lemma 1.1(2)]{AWZ2},
\[\det M_r = c \cdot \prod_{[\mu] \in \mathbb{P}(\Fp^m)} \left(\mu_1y_1^{p^r} + \cdots + \mu_my_m^{p^r}\right)\]
for some non-zero scalar $c \in\Fp$. Since $|\mathbb{P}(\Fp^m)| = (p^m - 1)/(p-1)$ and since $v_{|F}$ is a valuation, Lemma \ref{Vdet}(a) implies that
\[v(\det M_r) = (1 + p +  \cdots + p^{m-1})\lambda p^r.\]
In particular, $\det M_r$ is non-zero and $M_r$ is invertible.

By Cramer's rule, $(M_r^{-1})_{ij} \cdot \det M_r$ is (up to a sign) equal to the determinant of the matrix obtained from $M_r$ by removing the $j$-th row and $i$-th column. This determinant is a signed sum of monomials of the form
\[y_{i_1}^{p^r} y_{i_2}^{p^{r+1}}\cdots \widehat{y_{i_j}^{p^{r + j - 1}}} \cdots y_{i_m}^{p^{r+m-1}},\]
where the hat indicates that the factor $y_{i_j}^{p^{r + j - 1}}$ has been omitted. So
\[v\left( (M_r^{-1})_{ij} \cdot \det M_r\right) \geq (1 + p + \cdots + \widehat{p^{j-1}} + \cdots + p^{m-1})p^r\lambda\]
and the Proposition follows because $v_{|F}$ is a valuation.
\end{proof}

\subsection{The maps $\zeta^{(i)}_r$}
\label{Zeta}
Let us left-multiply the equation $(\ref{MatrixEq})$ by the inverse of $M_r$:
\[ M_r^{-1} \begin{pmatrix} \tau \varphi^{p^r} - \tau \\ \tau \varphi^{p^{r+1}} - \tau \\ \vdots \\ \tau \varphi^{p^{r+m-1}} - \tau \end{pmatrix} =  \begin{pmatrix} \tau \partial_1 \\ \tau \partial_2 \\ \vdots \\ \tau \partial_m \end{pmatrix} + M_r^{-1} N_r \mathbf{v} + M_r^{-1}\begin{pmatrix} \eta_r \\ \eta_{r+1} \\ \vdots \\ \eta_{r+m-1} \end{pmatrix}\]
and let $\zeta^{(i)}_r \in \mathcal{B}(kG,Q)$ be the $i$th element on the left hand side. Precisely,
\[\zeta^{(i)}_r := \sum_{j=1}^m (M_r^{-1})_{ij} (\tau\varphi^{p^{r + j-1}} - \tau).\]
We can now state the main technical result of $\S \ref{MainSec}$. 
\begin{thm} For each $i=1,\ldots,m$, the limit
\[\lim_{r \to \infty} \zeta^{(i)}_r\]
exists in $\mathcal{B}(kG,Q)$ and equals the operator $\tau\partial_i : kG \to Q$.
\end{thm}

We begin the proof with the following technical estimate.

\begin{lem}$\inf\limits_{|\alpha| \geq 2}
 \{p^r \lambda (|\alpha| - 1) + \deg \tau \qdel{\alpha}\} \geq \frac{1}{2} p^r \lambda$ for all $r \gg 0$.
\end{lem}
\begin{proof} Let $\omega_{\max} = \max\limits_{1 \leq i \leq d} \omega(g_i)$ and note that $\deg \qdel{\alpha} \geq -|\alpha| \omega_{\max}$ for all $\alpha \in \mathbb{N}^d$ by Theorem \ref{qdel}(d). Since $\lambda > 0$, we can find $s \geq 0$ such that $\lambda p^s > \omega_{\max}$. Suppose that $|\alpha|\geq 2$ and $r\geq s$; then $\lambda p^r - \omega_{\max} > 0$, so 
\[\begin{array}{lll} p^r \lambda(|\alpha|-1) + \deg \tau \qdel{\alpha} &\geq & p^r \lambda (|\alpha|-1) + \deg \tau - |\alpha| \omega_{\max}  \\
&=& (\lambda p^r - \omega_{\max})(|\alpha| - 1) + \deg \tau - \omega_{\max} \\
&\geq& \lambda p^r + \deg \tau - 2 \omega_{\max} \\
&=& \frac{1}{2}\lambda p^r + \left(\frac{1}{2} \lambda p^r + \deg \tau - 2\omega_{\max}\right).\end{array}
\]
The expression in the brackets on the right hand side is eventually positive, and the result follows. 
\end{proof}

\subsection{Proof of Theorem \ref{Zeta}}
Fix the index $i$ and write $\zeta_r = \zeta^{(i)}_r$. For each $\alpha\in\mathbb{N}^d$, define the ``coefficient''
\[C_{r,\alpha} := \sum_{j=1}^m (M_r^{-1})_{ij} \mathbf{y}^{p^{r + j - 1}\alpha} \in Q.\]
We may now write
\begin{equation}\label{MainEq}\zeta_r = \sum_{0 \neq \alpha \in \mathbb{N}^d} C_{r,\alpha} \tau\qdel{\alpha} + \epsilon_r,\end{equation}
where $\epsilon_r$ is the error term
\[\epsilon_r := \sum_{j=1}^m (M_r^{-1})_{ij} \sum_{\alpha \neq 0} \left(\tau \left(\langle \varphi^{p^{r+j-1}}, \qdel{\alpha} \rangle\right) - \mathbf{y}^{\alpha p^{r+j-1}} \right) \tau \qdel{\alpha} \in \mathcal{B}(kG,Q).\]
The definition of the matrix $M_r$ gives
\[C_{r, e_\ell} = \sum_{j=1}^m (M_r^{-1})_{ij} y_\ell^{p^{r + j - 1}} = \sum_{j=1}^m (M_r^{-1})_{ij} (M_r)_{j\ell} = \delta_{\ell i} := \left\{ \begin{array}{lll} 1 & \mbox{if} & \ell = i \\ 0 & \mbox{if} & \ell \neq i\end{array} \right.\]
provided that $1 \leq \ell \leq m$, whence
\begin{equation}\label{Z}\sum_{\ell=1}^m C_{r,e_\ell} \tau \qdel{e_\ell} = \tau \partial_i .\end{equation}
We now proceed to estimate the remaining terms of equation $(\ref{MainEq})$.  Applying Lemma \ref{Vdet}(b) and Proposition \ref{LinAlg} we see that
\[ v\left((M_r^{-1})_{ij} \mathbf{y}^{p^{r+j-1}\alpha}\right) \geq p^{r + j - 1}\lambda(|\alpha| - 1)\]
for any $\alpha \in \mathbb{N}^d$, any $j=1,\cdots,m$ and any $r \geq 0$. Therefore
\[ v(C_{r\alpha}) \geq p^r \lambda (|\alpha| - 1)\]
for all $r \geq 0$ and $0 \neq \alpha \in\mathbb{N}^d$, and Lemma \ref{Zeta} shows that
\begin{equation}\label{A} \inf\limits_{|\alpha|\geq 2} \deg ( C_{r\alpha} \tau \qdel{\alpha} ) \geq \frac{1}{2} p^r \lambda \end{equation}
for all $r \gg 0$. Next, by Propositions \ref{ModPExp} and \ref{LinAlg}, 
\[\begin{array}{lll} \deg(\epsilon_r) &\geq & \inf\limits_{1\leq j \leq m} \inf\limits_{\alpha\neq 0} \{- p^{r + j - 1}\lambda  + p^{2(r+j-1)-r_1} + p^{r+j-1}\lambda (|\alpha| - 1) + \deg \tau \qdel{\alpha}\} \\
&\geq & \inf\limits_{\alpha \neq 0} \{ p^{2r-r_1} - p^{r+m-1}\lambda + p^r \lambda(|\alpha| - 1) + \deg \tau \qdel{\alpha}\} \\

& \geq & p^{2r-r_1} - p^{r+m-1} \lambda  + \inf\{\deg \tau \partial_1, \cdots , \deg\tau \partial_d, \frac{1}{2}p^r \lambda\} \end{array}\]
for all $r \gg 0$, again using Lemma \ref{Zeta}. Thus
\begin{equation}\label{B} \deg(\epsilon_r) \geq p^{2r-r_1} - p^{r+m-1} \lambda  + C\end{equation}
for some constant $C$, for all $r \gg 0$. This exhausts the terms of $(\ref{MainEq})$, unless $m < d$ and $\alpha = e_\ell$ for some $\ell > m$. In this case, let $\mu := \min\{v(y_\ell) - \lambda : m < \ell \leq d\}$; then $\mu > 0$ by Lemma \ref{Vdet}(b) and for any $\alpha = e_\ell$ with $\ell > m$, we have
\[ v\left((M_r^{-1})_{ij} y_\ell^{p^{r+j-1}}\right) \geq p^{r+j-1} (v(y_\ell) - \lambda) \geq p^r \mu\]
by Proposition \ref{LinAlg}. Therefore 
\begin{equation}\label{C} \deg( C_{r,e_\ell} \tau \qdel{e_\ell} ) \geq p^r \mu + \deg \tau - \inf\limits_{i>m} \omega(g_i) \end{equation}
for all $m < \ell \leq d$ and all $r \geq 0$.  It now follows from the estimates  $(\ref{Z})$, $(\ref{A})$, $(\ref{B})$ and $(\ref{C})$ that
\[\deg (\zeta_r - \tau\partial_i) \to \infty\quad\mbox{as}\quad r \to \infty\]
and the Theorem is proved. \qed

\vspace{0.2cm}
We are now just one Lemma away from our proof of Theorem \ref{MainB}.

\subsection{Lemma}\label{Idemps}
Let $H g^\nu  = H g_1^{\nu_1} \cdots g_m^{\nu_m} $ be a coset of $H$ in $G$ for some $0 \leq \nu_i < p$, and let $\delta_{Hg^\nu}$ be its characteristic function. Then inside the ring $\mathcal{B}(kG)$,
\[ \rho(\delta_{Hg^\nu }) = \prod\limits_{i=1}^m \left(1 - (\partial_i - \nu_i)^{p-1}\right)\]
is a polynomial in the quantized divided powers $\partial_1,\ldots,\partial_m$.
\begin{proof} Both sides are left $kH$-module endomorphisms of $kG$, so it is enough to check that they agree on elements of the form $g^\mu := g_1^{\mu_1}\cdots g_m^{\mu_m}$, $0 \leq \mu_i < p$. Since 
\[1 - (a-b)^{p-1} = \delta_{ab}\quad\mbox{for any}\quad a,b \in \Fp\]
by Fermat's little theorem and $\partial_i(g^\mu) = \mu_i g^\mu$ by Theorem \ref{qdel}(a), we have
\[\prod\limits_{i=1}^m \left(1 - (\partial_i - \nu_i)^{p-1}\right) ( g^\mu ) = g^\mu \prod\limits_{i=1}^m (1 - (\mu_i - \nu_i)^{p-1})  = g^\mu \prod\limits_{i=1}^m \delta_{\nu_i \mu_i} = \rho(\delta_{H g^\nu})(g^\mu)\]
as required. \end{proof}

\subsection{Proof of Theorem \ref{MainB}} 
\label{PfCntrlThm}Using Theorem \ref{RescVal}, choose a filtration $v$ on the Goldie quotient ring $Q$ of $kG/P$ such that the natural map $\tau : kG \to Q$ is bounded and such that $v_{|F}$ is a valuation on the centre $F$ of $Q$. Since $\Aut_Z^\omega(G)$ is torsion-free by Corollary \ref{Sat}(b), $\varphi^{p^r}$ is non-trivial for any $r \geq 0$. In view of Lemma \ref{lambda}, we may assume that $z := z(\tilde{\varphi})$ sends $\Sat(G)$ into $Z$ and that $v(\tau(z(g)-1)) \geq 1$ for all $g \in G$.

Let $x \in P$, let $i = 1,\ldots, m$ and define $\zeta^{(i)}_r : kG \to Q$ as in $\S \ref{Zeta}$. Then $\zeta^{(i)}_r(x) = 0$ for all $r \geq 0$ because $\varphi$ preserves $P$ by assumption. Since $\zeta^{(i)}_r$ converges to $\tau \partial_i$ uniformly on $kG$ by Theorem \ref{Zeta}, it converges pointwise: $\tau \partial_i(x) = \lim_{r \to \infty} \zeta^{(i)}_r(x) = 0$. But $\ker \tau = P$ by definition, so  $\partial_i(x) \in P$ and therefore $\partial_i(P) \subseteq P$ for all $i = 1,\ldots, m.$ Since $\rho(\delta_{Hg})$ is a polynomial in the $\partial_i$ for each $Hg \in G/H$ by Lemma \ref{Idemps}, $P$ is a $(C^\infty)^H$-submodule of $kG$. It now follows from \cite[Lemma 2.9, Definition 2.6 and Proposition 2.8]{Ard2011} that $P = (P\cap kH)kG$. \qed

\section{Applications}

\subsection{Roseblade's Theorem D}\label{RoseD}
Let $A$ be a free abelian pro-$p$ group of finite rank and let $\Gamma$ be a closed subgroup of $\Aut(A)$. Lemma \ref{ContSub} implies that $\Gamma$ acts on $P^\chi / P^\dag$ for any proper $\Gamma$-invariant ideal $P$ of $kA$, and we write $\Gamma_P$ for the image of $\Gamma$ in $\Aut(P^\chi / P^\dag)$.

We begin our list of applications of Theorem \ref{MainB} with an exact analogue of Roseblade's \cite[Theorem D]{Roseblade}. 

\begin{thm} For any $\Gamma$-invariant prime ideal $P$ of $kA$, $\Gamma_P$ is finite.
\end{thm}
\begin{proof} Since $P^\dag$ is isolated by Lemma \ref{IsoSub}(c), by replacing $A$ by $A / P^\dag$ and $P$ by its image in $k[[A/P^\dag]]$, we may assume that $P$ is faithful. Now $P = (P \cap kP^\chi)kA$ by \cite[Theorem A]{Ard2011} and $P \cap kP^\chi$ is still a $\Gamma$-invariant prime of $kP^\chi$ because $kA$ is commutative, so we may assume that $P^\chi = A$. 

Let $\omega : A \to \mathbb{Z} \cup \{\infty\}$ be the standard $p$-valuation given by $\omega\left(A^{p^n} \backslash A^{p^{n+1}}\right) = n + 1$; then $\varphi \in \Aut^\omega(A)$ if and only if for all $n \geq 0$, $\varphi(a)a^{-1} \in A^{p^{n+1}}$ whenever $a \in A^{p^n}$. Hence $\Aut^\omega(A)$ is the kernel of the natural map $\Aut(A) \to \Aut(A/A^p)$. But $\Gamma_P \cap \Aut^\omega(A)$ is trivial by Theorem \ref{MainB}, so $\Gamma_P$ embeds into the finite group $\Aut(A/A^p)$. 
\end{proof}

We say that $\Gamma$ acts \emph{rationally irreducibly} on $A$ if every non-trivial $\Gamma$-invariant subgroup is open in $A$. If $\mathcal{L}(\Gamma)$ denotes the $\Qp$-Lie algebra of $\Gamma$ then this is equivalent to $\mathcal{L}(A)$ being an irreducible $\mathcal{L}(\Gamma)$-module. We can now prove \cite[Conjecture 5.1]{ArdWad2009}. 

\begin{cor} Suppose that $[\Gamma, A] \leq A^p$. 
\be
\item Every faithful $\Gamma$-invariant prime ideal $P$ of $kA$ is controlled by $A^\Gamma$. 
\item If $\Gamma$ acts rationally irreducibly on $A$ then the zero ideal and the maximal ideal are the only $\Gamma$-invariant prime ideals of $kA$.
\ee
\end{cor}
\begin{proof} (a) Let $\omega$ be the standard $p$-valuation on $A$; then $\Gamma$ is contained in $\Aut^\omega(A)$ by assumption. Let $\varphi \in \Gamma$; then $\varphi$ stabilizes $P^\chi$ and $\deg_{\omega_{|P^\chi}}(\varphi_{|P^\chi}) \geq \deg_\omega(\varphi)$ by definition. Hence $\Gamma_P$ is contained in $\Aut^{\omega_{|P^\chi}}(P^\chi)$, which is a torsion-free group by Corollary \ref{Sat}(b). So $\Gamma_P$ is trivial by the Theorem, whence $\Gamma$ fixes $P^\chi$ pointwise and $P^\chi \leq A^\Gamma$.

(b) Suppose first that $P^\dag$ is non-trivial. Then it contains $A^{p^n}$ for some $n$, being a $\Gamma$-invariant subgroup of $A$. But then $(a-1)^{p^n} \in P$ for all $a \in A$; since $P$ is prime we deduce that $P$ is the augmentation ideal of $kA$.  

Now suppose that $P^\dag = 1$. Then $P^\chi \leq A^\Gamma$ by part (a). Since $\Gamma$ acts rationally irreducibly on $A$, either $A^\Gamma$ is trivial or $A = \Zp$ and $\Gamma$ is trivial. In the first case $P^\chi = 1$ which forces $P = 0$, and in the second case $P = 0$ also because the only non-zero prime of $k\Zp \cong k[[t]]$ is the maximal ideal, which isn't faithful.
\end{proof}

Note the ``change of $p$-valuation" technique used in the proof of part (a). Corollary \ref{RoseD} is also of interest in connection with the mod-$p$ local Langlands programme: see \cite{HMS} for more details. In that paper, a special case of part (b) appears as \cite[Theorem 1.1]{HMS}.
\newpage
\subsection{Just infinite induced modules}\label{JustInf}
Recall the definition of \emph{just infinite modules} from $\S \ref{JustInfIntro}$.

\begin{thm} Let $A$ be a free abelian pro-$p$ group of finite rank, let $\Gamma$ be a closed subgroup of $\Aut(A)$ and let $G = A \rtimes \Gamma$ be the semi-direct product. If $[\Gamma, A] \leq A^p$ and $\Gamma$ acts rationally irreducibly on $A$, then the induced module $k \otimes_{k\Gamma} kG$ is just infinite.
\end{thm}
\begin{proof} Let $\pi : kG \twoheadrightarrow M:=k \otimes_{k\Gamma} kG$ be the map $x \mapsto 1 \otimes x$ and let $N$ be a non-zero $kG$-submodule of $M$. Then $\pi^{-1}(N)$ is a right ideal of $kG$ since $\pi$ is right $kG$-linear. Let $x \in \pi^{-1}(N)$ and $\gamma \in \Gamma$; then
\[ \pi(\gamma x \gamma^{-1}) = 1 \otimes \gamma x \gamma^{-1} = 1.\gamma \otimes x \gamma^{-1} = 1 \otimes x \gamma^{-1} = \pi(x) \gamma^{-1} \in N,\]
so $\pi^{-1}(N)$ is a $\Gamma$-invariant right ideal of $kG$. Since $A$ is stable under conjugation by $\Gamma$ inside $G$, $I = \pi^{-1}(N) \cap kA$ is a $\Gamma$-invariant right ideal of $kA$ and $I$ is non-zero because the restriction of $\pi$ to $kA$ is bijective by construction. Furthermore $I$ is two-sided since $A$ is abelian.

Let $P$ be a minimal prime ideal above $I$; since $I$ is $\Gamma$-invariant, $P$ is $\Gamma$-orbital so its stabilizer $S$ has finite index in $\Gamma$. Therefore $\mathcal{L}(S) = \mathcal{L}(\Gamma)$ so $S$ still acts rationally irreducibly on $A$ which forces $P$ to be the maximal ideal $\mathfrak{m}$ of $kA$ by Corollary \ref{RoseD}(b). Hence the prime radical $\sqrt{I}$ of $I$ is equal to $\mathfrak{m}$ and therefore $I$ contains some power of $\mathfrak{m}$ which has finite codimension in $kA$. The result follows since $\pi$ induces a $k$-linear bijection $kA/I \stackrel{\cong}{\longrightarrow} M/N$.
\end{proof}
We now present another example of a just infinite induced module, arising from split semisimple groups.

\subsection{Proof of Theorem \ref{MainD}} \label{ParaVerma}
Let $\mathfrak{g} = \mathfrak{sl}_n(\Qp) = \mathcal{L}(G)$, let $\mathfrak{p}_-$ be the opposite parabolic to $\mathfrak{p}$ and let $\mathfrak{a}$ be its nilradical. Note that $\mathfrak{a}$ is abelian because $\mathfrak{g} = \mathfrak{sl}_n(\Qp)$ and $\mathfrak{p}$ is a maximal parabolic. Following \cite[\S II.1.8]{Jantzen} we call $\mathfrak{l} := \mathfrak{p} \cap \mathfrak{p}_-$ the \emph{standard Levi factor} of $\mathfrak{p}$. Then we have a semi-direct product decomposition
\[\mathfrak{p}_- = \mathfrak{a} \rtimes \mathfrak{l} \]
and a vector space decomposition
\[\mathfrak{g} = \mathfrak{a} \oplus \mathfrak{p}.\]
If $P_-$, $A$ and $L$ denote the corresponding uniform subgroups of $G$, these vector space decompositions imply that $P P_- = G$ and and $P_- \cap P = L$. Therefore there is an isomorphism 
\[ L \backslash P_- \stackrel{\cong}{\longrightarrow} P \backslash G\]
of right $P_-$-spaces, which induces an isomorphism
\[ k \otimes_{kL} kP_- \stackrel{\cong}{\longrightarrow} k \otimes_{kP} kG\]
of right $kP_-$-modules. Since every $kG$-submodule of $k\otimes_{kP} kG$ is also a $kP_-$-submodule, it is now enough to prove that $k \otimes_{kL} kP_-$ is just infinite as a $kP_-$-module. 

Consider the adjoint action of $\mathfrak{l}$ on $\mathfrak{a}$. Using the natural representation of $\mathfrak{g} = \mathfrak{sl}_n(\Qp)$ and the assumption that $\mathfrak{p}$ is a \emph{maximal} parabolic, we can identify $\mathfrak{l}$ with the reductive Lie subalgebra $\left(\mathfrak{gl}_m(\Qp) \times \mathfrak{gl}_{n-m}(\Qp)\right) \cap \mathfrak{g}$ of $\mathfrak{g}$ for some $1 \leq m \leq n-1$. Then $\mathfrak{l}$ contains the block diagonal subalgebra $\mathfrak{d} := \mathfrak{sl}_m(\Qp) \times \mathfrak{sl}_{n-m}(\Qp)$. If $V_r$ denotes the natural irreducible representation of $\mathfrak{sl}_r(\Qp)$, then $\mathfrak{a}$ is isomorphic to $\Hom_{\Qp}(V_{n-m}, V_m) \cong V_m \otimes V_{n-m}^\ast $ as a $\mathfrak{d}$-module and is therefore irreducible as such. Therefore the adjoint representation of $\mathfrak{l}$ on $\mathfrak{a}$ is also irreducible: $\mathfrak{a}$ is a minimal abelian ideal of $\mathfrak{p}_-$.

Now $P_- = A \rtimes L$ is a semi-direct product, and $L$ acts rationally irreducibly on $A$ by the above. Because $L$ normalizes $A$, $G$ is uniform and $A$ is isolated in $G$,  $[L,A] \leq A \cap [G,G] \leq A^p$ inside $P_-$, so the $kP_-$-module $k \otimes_{kL} kP_-$ is just infinite by Theorem \ref{JustInf}.\qed
\subsection{Zalesskii's Theorem}\label{Zal}
\begin{lem} Let $G$ be a complete $p$-valued group of finite rank. Then
\be \item $Z(G)$ is isolated in $G$, and
\item $Z(U) = Z(G) \cap U$ for any open subgroup $U$ of $G$. 
\ee \end{lem}
\begin{proof} (a) Let $g \in G$ be such that $g^{p^n} \in Z(G)$ for some $n$. Then $(x^{-1}gx)^{p^n} = g^{p^n}$ for all $x \in G$ and it follows from \cite[Proposition III.2.1.4]{Laz1965} that $g \in Z(G)$. 

(b) Let $g \in Z(U)$. Since $U$ contains $G^{p^n}$ for some large enough $n$, $(g^{-1}xg)^{p^n} = x^{p^n}$ for all $x \in G$. As in part (a), we deduce that $g \in Z(G)$ so $Z(U) \subseteq Z(G) \cap U$. The reverse inclusion is trivial.
\end{proof}
We can finally give a proof of our analogue of Zalesskii's Theorem.

\begin{thm} Let $G$ be a nilpotent complete $p$-valued group of finite rank with centre $Z$. Then every faithful prime ideal $P$ of $kG$ is controlled by $Z$.
\end{thm}
\begin{proof} By Theorem \ref{GammaOrb} applied with $A = Z$, it is enough to show that every faithful virtually non-splitting right ideal $I$ of $kG$ is controlled by $Z$. Now $I = (I \cap kU)kG$ for some open subgroup $U$ of $G$ with $I \cap kU$ a non-splitting prime of $kU$. Since $Z$ contains $Z(U)$ by Lemma \ref{Zal}(b), it is enough to prove that $I \cap kU$ is controlled by $Z(U)$. So by replacing $G$ by $U$ and $P$ by $I \cap kU$ we may further assume that our faithful prime ideal $P$ is \emph{non-splitting}.

Let $H := P^\chi$ and define $K = C_G(H/Z(H)) = \{ g \in G : [g,H] \leq Z(H)\}$ to be the centralizer in $G$ of $H / Z(H)$. Then $K$ contains the centralizer $C_G(H)$ of $H$ in $G$ and $\Gamma := K/C_G(H)$ acts faithfully on $H$ by group automorphisms; since $K$ is $p$-valued and acts on $H$ by automorphisms which are trivial mod $Z(H)$, we will identify $\Gamma$ with a subgroup of $\Aut^\omega_{Z(H)}(H)$. 

Let $Q = P \cap kH$; then $P = QkG$ by \cite[Theorem A]{Ard2011}. Since $H$ is the controller subgroup of $P$, we see that $Q$ cannot be controlled by any proper subgroup of $H$. On the other hand, since $P$ is non-splitting, $Q$ is a prime ideal of $kH$ by Proposition \ref{NonSplit}. Since $\Gamma$ preserves $Q$, Theorem \ref{MainB} implies that $\Gamma$ must be trivial and therefore $K = C_G(H)$.

Now the definition of $K$ shows that $K \cap H$ is the second term $Z_2(H)$ in the upper central series of $H$. On the other hand, since $K = C_G(H)$ this intersection is just the centre of $H$. Because $H$ is nilpotent, it must actually be abelian. Inspecting the definition of $K$ again gives $K = N_G(H) = G$, but also $K = C_G(H)$ and therefore $G$ centralizes $H$. In other words, $H$ is central in $G$ and $P$ is controlled by $Z$.
\end{proof}

\subsection{A completed crossed product}\label{CompCross}
Our final application of Theorem \ref{MainB} is that when $G$ is nilpotent, every prime ideal $P$ of $kG$ is \emph{completely prime}, that is, $kG/P$ has no zero-divisors. This will require a little preparation.

\begin{lem}Let $G$ be a complete $p$-valued group of finite rank and let $N$ be a closed isolated normal subgroup of $G$. Then we can find $c_1,\ldots, c_e \in kG$ such that 
\be\item every element of $kG$ can be written uniquely as a non-commutative formal power series in $c_1,\ldots, c_e$ with coefficients in $kN$:
\[ kG = \left\{\sum_{\gamma \in\mathbb{N}^e}r_\gamma \C{\gamma} : r_\gamma \in kN \mb{for all} \gamma \in \mathbb{N}^e\right\},\]
\item the valuation $w$ on $kG$ satisfies
\[w\left(\sum_{\gamma \in \mathbb{N}^e} r_\gamma \C{\gamma} \right) = \inf_{\gamma\in \mathbb{N}^e} w(r_\gamma) + w(\C{\gamma}).\]
\ee\end{lem}
\begin{proof} Let $d = \dim G$ and $e = \dim(G/N)$. By Lemma \ref{OrdBas}, we can choose an ordered basis $\{g_1,\ldots, g_d\}$ for $G$ such that $\{g_1^{p^{n_1}},\ldots, g_{d-e}^{p^{n_{d-e}}}\}$ is an ordered basis for $N$ for some integers $n_1 \leq n_2 \leq \cdots \leq n_{d-e}$. Since $G/N$ is torsion-free, \cite[IV.3.4.2]{Laz1965} implies that the quotient filtration on $G/N$ is a $p$-valuation, so $\gr (G/N)$ has no $\pi$-torsion. Since $\gr G / \gr N$ naturally embeds into $\gr(G/N)$ by \cite[II.1.1.8.3]{Laz1965}, $n_1=n_2= \cdots = n_{d-e} = 0$ and $\{g_1,\ldots, g_{d-e}\}$ is an ordered basis for $N$.

Let $c_i = g_{d-e+i} - 1 \in kG$ for all $i = 1,\ldots, e$ and let $Y_i = \gr^w(c_i)$ be the principal symbol of $c_i$ in $\gr^w kG$. Then Lemma \ref{valkG}(a) implies that
\[\gr^wkG \cong (\gr^w kN)[Y_1,\ldots, Y_e]\]
so $\{\gr^w \C{\gamma} : \gamma \in \mathbb{N}^e\}$ is a free generating set for $\gr^wkG$ as a $\gr^wkN$-module. The result now follows from \cite[Th\'eor\`eme I.2.3.17]{Laz1965}.
\end{proof}
Thus $kG \cong kN[[c_1,\ldots,c_e]]$ as a $kN$-module. Because of this result, it is tempting to think of $kG$ as a kind of ``completed crossed product" of $kN$ with $G/N$. But we will not develop this intuition any further.

\subsection{Theorem}
\label{CompPrime}
Let $G$ be a complete $p$-valued group of finite rank with centre $Z$, and let $P$ be a prime ideal of $kZ$. 
\be \item $PkG$ is completely prime.
\item If $P$ is faithful, then so is $PkG$. \ee
\begin{proof} (a) Let $\tau : kZ \twoheadrightarrow kZ/P$ be the natural projection and let $Q$ be the field of fractions of $kZ/P$. Applying Theorem \ref{RescVal} to the group $Z$ and the prime ideal $P$ of $kZ$, we obtain a valuation $v : Q \to \mathbb{R}_\infty$ such that $v \tau(x) \geq w(x)$ for all $x \in kZ$. This valuation is separated because $Q$ is a field.

Since $Z$ is isolated in $G$ by Lemma \ref{Zal}(a), we can apply Lemma \ref{CompCross} to find $c_1,\ldots, c_e \in kG$ such that $kG \cong kZ[[c_1,\ldots,c_e]]$ as a $kZ$-module. Now define a function $f : kG \to \mathbb{R}_\infty$ by the rule
\[f \left( \sum_{\gamma \in \mathbb{N}^e} r_\gamma \C{\gamma} \right) :=  \inf\limits_{\gamma \in \mathbb{N}^e} v \tau(r_\gamma) + \wth{\gamma}. \]
We claim that $f$ is a ring filtration on $kG$. To see this, consider the product $\C{\alpha} \C{\beta}$ inside $kG$ for $\alpha, \beta \in \mathbb{N}^e$; using Lemma \ref{CompCross}(a) we can rewrite it as 
\[ \C{\alpha} \C{\beta} = \sum_{\gamma \in \mathbb{N}^e} \eta^{\alpha\beta}_\gamma \C{\gamma}\]
for some $\eta^{\alpha\beta}_\gamma \in kZ$, and these coefficients satisfy
\[ w(\C{\alpha} \C{\beta}) = \inf\limits_{\gamma \in \mathbb{N}^e} w(\eta^{\alpha\beta}_\gamma) + \wth{\gamma} \]
by Lemma \ref{CompCross}(b). Now
\[ w(\C{\alpha} \C{\beta}) = \wth{\alpha} + \wth{\beta} \mb{for all} \alpha,\beta\in \mathbb{N}^e \]
because $w$ is a valuation on $kG$. Since $v\tau(x) \geq w(x)$ for all $x \in kZ$, we see that
\begin{equation}\label{EtaEst} v\tau(\eta^{\alpha\beta}_\gamma) + \wth{\gamma} \geq \wth{\alpha} + \wth{\beta}\mb{for all} \alpha,\beta,\gamma \in \mathbb{N}^e. \end{equation}
Let $r = \sum_{\alpha\in \mathbb{N}^e} r_\alpha\C{\alpha} \in kG$ and $s = \sum_{\beta \in \mathbb{N}^e} s_\beta \C{\beta} \in kG$; then
\[r s = \sum_{\gamma \in \mathbb{N}^e} \left(\sum_{\alpha,\beta \in \mathbb{N}^e} r_\alpha s_\beta \eta^{\alpha\beta}_\gamma  \right) \C{\gamma}.\]
Applying the definition of $f$ and equation $(\ref{EtaEst})$, we obtain
\[ \begin{array}{lll} f(r s) &=& \inf\limits_{\gamma \in \mathbb{N}^e} v \tau \left(\sum_{\alpha,\beta \in \mathbb{N}^e} r_\alpha s_\beta \eta^{\alpha\beta}_\gamma  \right) + \wth{\gamma} \\ 
&\geq& \inf\limits_{\gamma \in \mathbb{N}^e} \left(\inf\limits_{\alpha,\beta \in \mathbb{N}^e} v \tau (r_\alpha) + v\tau(s_\beta) + v\tau(\eta^{\alpha\beta}_\gamma)\right) + \wth{\gamma} \\
& \geq & \inf\limits_{\alpha,\beta \in \mathbb{N}^e} v\tau(r_\alpha) + v\tau(s_\beta) + \wth{\alpha} + \wth{\beta} \\
& \geq & f(r) + f(s).\end{array}\]
The inequality $f(r + s) \geq \min \{f(r), f(s)\}$ is easy to verify, so $f$ is indeed a ring filtration on $kG$, satisfying
\[f(r) \geq w(r) \mb{for all} r \in kG\]
by Lemma \ref{CompCross}(b). 

Next, equip $kZ$ with the valuation $v\tau$; this valuation is not separated and in fact $(v\tau)^{-1}(\infty) = P$ because $v$ is a separated valuation on $Q$. Since $f(r) = v\tau(r)$ for all $r \in kZ$, the natural map $(kZ, v\tau) \to (kG, f)$ is strictly filtered and induces an inclusion of associated graded rings
\[ \iota :\gr^{v\tau} kZ \hookrightarrow \gr^f kG.\]
Since $\gr^w kG$ is commutative and $w(c_i) = f(c_i)$ by definition, we have
\[f([c_i,c_j]) \geq w([c_i,c_j]) > w(c_i) + w(c_j) = f(c_i) + f(c_j)\]
for all $i,j = 1,\ldots, e$ which shows that the principal symbols $\gr^f c_i$ of the $c_i$'s commute pairwise. Therefore we obtain a ring homomorphism
\[ \psi : \left(\gr^{v \tau}kZ\right)[Y_1,\ldots,Y_e] \to \gr^f kG\]
which extends $\iota$ and sends $Y_i$ to $\gr^f c_i$. Using the definition of $f$, it is straightforward to verify that $\psi$ is actually a bijection. Because $v\tau$ is a valuation, $\gr^{v \tau}kZ$ is a domain so $\gr^f kG$ is a domain, which implies that $f^{-1}(\infty)$ is a completely prime ideal of $kG$. 

Finally, by choosing a finite generating set for $P$ as an ideal, it is easy to see that 
\begin{equation}\label{PkG} PkG = \left\{\sum_{\gamma \in\mathbb{N}^e}r_\gamma \C{\gamma} : r_\gamma \in P \mb{for all} \gamma \in \mathbb{N}^e\right\}.\end{equation}
Since $P = (v\tau)^{-1}(\infty)$, it follows that $PkG = f^{-1}(\infty)$ is completely prime.

(b) Let $h\in (P kG)^\dag$. Then $h = zg_{d-e+1}^{\alpha_1} \cdots g_d^{\alpha_e}$ for some $z \in Z$ and $\alpha \in \Zp^e$, and 
\[ \begin{array}{lll} PkG \ni h - 1 &=& z (1 + c_1)^{\alpha_1} \cdots (1 + c_e)^{\alpha_e} - 1  \\
&=& (z - 1) + \sum_{\gamma \neq 0} z \binom{\alpha}{\gamma} \C{\gamma} 
\end{array}\]
by the binomial expansion. Applying $(\ref{PkG})$, we deduce that $z - 1  \in P$ and $z \binom{\alpha}{\gamma} \in P$ for all $0 \neq \gamma \in \mathbb{N}^e$. The first constraint forces $z = 1$ as $P$ is a faithful prime by assumption. Because the prime ideal $P$ is proper, we must have $\binom{\alpha}{\gamma} = 0$ in the field $k$ for all $0 \neq \gamma \in \mathbb{N}^e$.  By applying Mahler's Theorem \ref{Mahler}, we see that every locally constant function $f : \Zp^e \to k$ satisfies $f(\alpha) = f(0)$. But locally constant functions on $\Zp^e$ separate points, so $\alpha = 0$. Therefore $h = 1$ and $PkG$ is faithful.
\end{proof}

\begin{cor} Let $G$ be a nilpotent complete $p$-valued group of finite rank. Then every prime ideal $P$ of $kG$ is completely prime.
\end{cor}
\begin{proof} The normal subgroup $P^\dag$ of $G$ is isolated by Lemma \ref{IsoSub}(c), so $G/P^\dag$ is again a complete $p$-valued group of finite rank by \cite[IV.3.4.2]{Laz1965}. By replacing $G$ by $G/P^\dag$ we may therefore assume that our prime ideal $P$ is faithful. Now $P = (P\cap kZ)kG$ by Theorem \ref{Zal} and $P \cap kZ$ is a prime ideal in $kZ$ since $Z$ is the centre of $G$, so the result follows from Theorem \ref{CompPrime}(a).
\end{proof}

\subsection{Proof of Theorem \ref{MainA}}\label{PfMainA}
(a) Let $P = \Theta(Q) = \widetilde{Q} kG$; then $P \cap k\widetilde{N} = \widetilde{Q}$ by Lemma \ref{ComplGpRng}(b). Clearly $N \leq (\widetilde{Q})^\dag \leq P^\dag$; on the other hand if $g \in P^\dag$ then $gN \in \left(Q k[[G/N]]\right)^\dag$ which is the trivial group by Theorem \ref{CompPrime}(b) since $Q$ is faithful. So $P^\dag = N$ and therefore
\[ \Psi(\Theta(Q)) = \frac{P\cap k\widetilde{P^\dag}}{(P^\dag - 1)k\widetilde{P^\dag}} = \frac{\widetilde{Q}}{(N - 1)k\widetilde{N}} = Q.\]

(b) Let $Q = \Psi(P) = \frac{P \cap k\widetilde{P^\dag}}{(P^\dag - 1)k\widetilde{P^\dag}}$ so that $\widetilde{Q} = P \cap k\widetilde{P^\dag}$. Because $G$ is nilpotent, the image $P / (P^\dag - 1)kG$ of $P$ in $k[[G/P^\dag]]$ is controlled by $Z_{P^\dag}$ by Theorem \ref{Zal}, so 
\[P = (P \cap k\widetilde{P^\dag})kG = \widetilde{\Psi(P)}kG = \Theta(\Psi(P)).\]
Every ideal in $kG$ of the form $\Theta(Q)$ is completely prime by Theorem \ref{CompPrime} (a).
\bibliography{references}
\bibliographystyle{plain}
\end{document}